\documentclass[11pt]{article}
\usepackage[margin=.6in]{geometry}
\usepackage{amsmath,amssymb,amsthm,mathtools,bm}
\usepackage{booktabs,array}
\usepackage{graphicx}
\usepackage{microtype}
\usepackage{enumitem}
\usepackage{hyperref}
\usepackage[mathlines]{lineno}
\usepackage{xcolor}
\hypersetup{colorlinks=true,linkcolor=blue,citecolor=blue,urlcolor=blue}
\graphicspath{{figures/}}
\allowdisplaybreaks

\newtheorem{theorem}{Theorem}[section]
\newtheorem{lemma}[theorem]{Lemma}
\newtheorem{proposition}[theorem]{Proposition}
\newtheorem{corollary}[theorem]{Corollary}
\newtheorem{assumption}[theorem]{Assumption}
\theoremstyle{definition}

\newcommand{\R}{\mathbb{R}}
\newcommand{\Z}{\mathbb{Z}}
\newcommand{\cD}{\mathcal{D}}
\newcommand{\cL}{\mathcal{L}}
\newcommand{\cR}{\mathcal{R}}
\newcommand{\diag}{\operatorname{diag}}

\newcommand{\norm}[1]{\left\lVert #1\right\rVert}
\newcommand{\one}{\mathbf{1}}
\newcommand{\U}{\bm U}
\newcommand{\uvec}{\bm u}
\newcommand{\PhiV}{\bm\Phi}
\newcommand{\jj}{\bm j}
\newcommand{\kk}{\bm k}
\newcommand{\qq}{\bm q}
\newcommand{\ee}{\bm e}

\title{Structure-Preserving Numerical Schemes for Two-Stage Local and Nonlocal Dispersal Systems}
\author{Markjoe O. Uba\\\small School of Mathematical and Statistical Sciences, Northern Illinois University, DeKalb, IL 60115, USA}
\date{}

\begin{document}
\maketitle

\begin{abstract}
We construct and analyze structure-preserving numerical schemes for a two-stage local--nonlocal dispersal system. On general bounded connected smooth habitats, we prove fixed-$\delta$ spatial consistency. We introduce a semi-implicit scheme that is uniquely solvable for every $\Delta t>0$, preserves nonnegativity, and exactly retains the sign of the semidiscrete spectral threshold. A negative threshold yields geometric extinction of the numerical solution, while a positive threshold makes the zero state of the numerical scheme linearly unstable. On rectangular habitats, we construct a reflected Cartesian discretization that is asymptotically compatible with the Neumann local limit, with convergence estimates uniform with respect to the ratio between mesh size and interaction scale. Numerical experiments illustrate the structural, threshold, and local-limit behavior.
\end{abstract}


\section{Introduction}

Two-stage population models distinguish juveniles from reproductively mature adults and allow the two classes to have different movement, mortality, maturation, and reproduction rates, together with density-dependent losses within and between the two stages.  Consistent with the continuous models in \cite{CCM,CCS,Onyido2023,Onyido2024}, we consider a bounded connected biological habitat $\Omega\subset\R^d$, $d\ge1$, and the two-stage system
\begin{equation}
\label{eq:continuous-delta}
\begin{aligned}
\partial_t u_1
 &=\mu_1\cD_\delta u_1+r u_2-(a+s)u_1-bu_1^2-\tau c u_1u_2,\\
\partial_t u_2
 &=\mu_2\cD_\delta u_2+s u_1-eu_2-fu_2^2-\tau g u_1u_2,
\end{aligned}
\qquad x\in\Omega,\ t>0.
\end{equation}
Here $\mu_1,\mu_2>0$ are dispersal rates and $\tau\ge0$ measures the strength of cross-stage competition.  For $\delta>0$, $\cD_\delta$ denotes nonlocal dispersal within the habitat.  When the local model is considered, $\cD_0$ is the Laplacian with homogeneous Neumann boundary condition.  The local reaction--diffusion and nonlocal two-stage models were studied analytically in \cite{CCM,CCS} and \cite{Onyido2023,Onyido2024}, respectively; in those settings, the sign of the principal spectral quantity of the linearization at the zero solution determines the persistence--extinction threshold.

We use two spatial discretizations.  On a general smooth habitat,
Sections~\ref{sec:framework}--\ref{sec:scheme} use the strictly positive
symmetric kernel class in \eqref{eq:kernel-general}; this yields strictly
positive cell-to-cell interaction coefficients and gives the structural and
threshold results on geometry-conforming meshes. On rectangular habitats,
Section~\ref{sec:cartesian} uses a compactly supported reflected kernel,
positive near the origin, in order to recover the Neumann Laplacian with a
uniform local-limit estimate.  Combined with the semi-implicit time
discretization, these constructions yield structure-preserving schemes that
retain nonnegativity and the semidiscrete persistence--extinction threshold.

\subsection*{Contributions}
The main contributions are as follows.
\begin{enumerate}[label=(C\arabic*),leftmargin=2.8em]
\item On bounded connected smooth habitats, we prove that the discrete nonlocal dispersal operator has strictly positive off-diagonal entries, preserves constants and discrete mass, is self-adjoint and dissipative in the cell-mass inner product, and has nullspace consisting exactly of the constant vectors.
\item For every fixed $\delta>0$, we prove spatial consistency of the habitat-conforming dispersal operator, with an explicit bound in terms of the modulus of continuity of the approximated function.
\item We prove that the semi-implicit Euler scheme is uniquely solvable for every $\Delta t>0$ and preserves both nonnegative population densities and semidiscrete equilibria.
\item For every positive time step, we prove exact preservation of the semidiscrete persistence--extinction threshold.
\item We derive residual-based a posteriori lower and upper bounds for the continuous principal spectrum point from positive approximate eigenfunctions.
\item On rectangular habitats, we prove asymptotic compatibility with the Neumann local limit, second-order convergence of the discrete spectral threshold in mesh size and interaction scale, and finite-time convergence of the nonlinear scheme with first-order accuracy in time and second-order accuracy in mesh size and interaction scale, with constants independent of the ratio $h/\delta$.
\end{enumerate}

The paper is organized as follows.  Section~\ref{sec:framework} formulates the continuous two-stage model and its spectral threshold on a general habitat.  Section~\ref{sec:habitat-discretization} defines the habitat-conforming spatial discretization and proves its structural and fixed-$\delta$ consistency properties.  Section~\ref{sec:scheme} constructs the semidiscrete and semi-implicit two-stage schemes and states the threshold results.  Section~\ref{sec:cartesian} treats rectangular habitats and develops the reflected Cartesian local-limit analysis.  Section~\ref{sec:numerics} presents experiments on a smooth nonrectangular habitat and on a rectangular habitat.  The remaining proofs are given in Section~\ref{sec:proof-main}.

\section{Continuous model on a general habitat and spectral threshold}
\label{sec:framework}

Throughout Sections~\ref{sec:framework}--\ref{sec:scheme}, let $\Omega\subset\R^d$ be a bounded connected domain with $C^2$ boundary.

\subsection{Continuous system, coefficients, and ordering}

Set $X=C(\overline\Omega)$ with $\norm{u}_\infty=\sup_{x\in\overline\Omega}|u(x)|$, and define
\[
 X_+=\{u\in X:u(x)\ge0\ \text{for every }x\in\overline\Omega\},
 \qquad
 X_{++}=\{u\in X_+:\inf_{x\in\overline\Omega}u(x)>0\}.
\]
We use boldface letters for two-stage states $\uvec=(u_1,u_2)\in X\times X$ with
$\norm{\uvec}=\max\{\norm{u_1}_\infty,\norm{u_2}_\infty\}$.

\begin{assumption}[Reaction coefficients]
\label{ass:coeff}
The functions $a,b,c,e,f,g,r,s$ belong to $C^2(\overline\Omega)$ and are nonnegative, $r$ and $s$ are non-identically zero, and $b$ and $f$ are strictly positive on $\overline\Omega$.  The dispersal rates satisfy $\mu_1,\mu_2>0$, and the cross-stage competition parameter satisfies $\tau\ge0$.
\end{assumption}
Write $H=a+s$ and define
\begin{equation}
\label{eq:A-matrix}
 A(x)=
 \begin{pmatrix}
 -H(x)&r(x)\\ s(x)&-e(x)
 \end{pmatrix}.
\end{equation}

\subsection{Nonlocal dispersal on the habitat}

For each fixed $\delta>0$, assume
\begin{equation}
\label{eq:kernel-general}
 \kappa_\delta\in C(\overline\Omega\times\overline\Omega),
 \qquad
 \kappa_\delta(x,y)=\kappa_\delta(y,x)>0
 \quad (x,y\in\overline\Omega),
\end{equation}
and define
\begin{equation}
\label{eq:general-nonlocal}
 (\cD_\delta v)(x)
 =
 \int_\Omega
 \kappa_\delta(x,y)\big(v(y)-v(x)\big)\,dy,
 \qquad v\in X.
\end{equation}
This is the difference-form dispersal operator used in the continuous nonlocal two-stage model; compare \cite{Onyido2023,Onyido2024}.  The parameter $\delta$ indexes the family of nonlocal kernels; no local-limit scaling is imposed on \eqref{eq:kernel-general} in the general-habitat analysis. The strict positivity in \eqref{eq:kernel-general} is the standing kernel hypothesis throughout Sections~\ref{sec:framework}--\ref{sec:scheme}.  The compactly supported reflected family introduced in Section~\ref{sec:cartesian} is a second kernel construction designed specifically for the Neumann local limit.

\begin{lemma}[Continuous dispersal identities]
\label{lem:continuous-dispersal-identities}
For every $\delta>0$ and $v\in X$,
\[
 \cD_\delta 1=0,
 \qquad
 \int_\Omega \cD_\delta v(x)\,dx=0.
\]
Moreover,
\begin{equation}
\label{eq:continuous-energy}
 -\int_\Omega v(x)\cD_\delta v(x)\,dx
 =
 \frac12\int_\Omega\int_\Omega
 \kappa_\delta(x,y)\big(v(y)-v(x)\big)^2\,dy\,dx
 \ge0.
\end{equation}
\end{lemma}

\begin{proof}
The identity $\cD_\delta1=0$ follows directly from \eqref{eq:general-nonlocal}.  Symmetry of $\kappa_\delta$ and interchange of $x$ and $y$ give
\[
 \int_\Omega\int_\Omega
 \kappa_\delta(x,y)\big(v(y)-v(x)\big)\,dy\,dx=0.
\]
Applying the same symmetry after multiplying by $v(x)$ yields \eqref{eq:continuous-energy}.
\end{proof}

When the local model is considered, set
\begin{equation}
\label{eq:general-local}
 \cD_0=\Delta,
 \qquad
 \partial_\nu v=0\quad\text{on }\partial\Omega.
\end{equation}

\subsection{Linearization and persistence threshold}

The linearization of \eqref{eq:continuous-delta} at zero is
\begin{equation}
\label{eq:lin-cont}
 \partial_t\PhiV=\cL_\delta\PhiV,
 \qquad
 \cL_\delta=
 \begin{pmatrix}
 \mu_1\cD_\delta-H&r\\
 s&\mu_2\cD_\delta-e
 \end{pmatrix}.
\end{equation}
For $\delta>0$, define the principal spectrum point of $\cL_\delta$ on $X\times X$ by
\begin{equation}
\label{eq:lambda-cont}
 \lambda_\delta
 :=\sup\{\Re z:z\in\sigma(\cL_\delta)\}.
\end{equation}
The supremum is used here because $\cL_\delta$ acts on the
infinite-dimensional space $X\times X$, so the spectral bound need not
a priori be attained by a spectral value.
For the local operator obtained from \eqref{eq:lin-cont} and \eqref{eq:general-local}, write $\cL_0$ and denote its principal eigenvalue by $\lambda_0$.

Under the standing kernel and coefficient hypotheses of \cite{Onyido2024}, Theorem~2.1 identifies the lower and upper generalized principal eigenvalues with the principal spectrum point.  Hence
\begin{equation}
\label{eq:CW-cont}
 \lambda_\delta
 =\sup_{\PhiV\in X_{++}\times X_{++}}
 \inf_{x\in\overline\Omega,\,i}
 \frac{(\cL_\delta\PhiV)_i(x)}{\Phi_i(x)}
 =\inf_{\PhiV\in X_{++}\times X_{++}}
 \sup_{x\in\overline\Omega,\,i}
 \frac{(\cL_\delta\PhiV)_i(x)}{\Phi_i(x)}.
\end{equation}
For the continuous two-stage population model, $\lambda_\delta\le0$ implies extinction, whereas $\lambda_\delta>0$ implies persistence; see \cite[Theorem~1]{Onyido2023}.  Under the additional assumption $r,s\in X_{++}$, the same result yields persistence uniformly in space.

\section{Habitat-conforming spatial discretization}
\label{sec:habitat-discretization}

We use a conforming cell partition of $\Omega$; general-mesh discretizations
of nonlocal problems are analyzed in \cite{TianDu2013,TianDuDG2015}.  Let $\mathcal T_h=\{K_i\}_{i=1}^{N_h}$ be a conforming partition of $\Omega$ into cells of positive measure such that
\begin{equation}
\label{eq:mesh-partition}
 \overline\Omega
 =\bigcup_{i=1}^{N_h}\overline{K_i},
 \qquad
 K_i^\circ\cap K_j^\circ=\varnothing
 \quad(i\ne j),
 \qquad
 \max_{1\le i\le N_h}\operatorname{diam}(K_i)\le h.
\end{equation}
Set
\[
 m_i=|K_i|,\qquad
 M_h=\diag(m_1,\ldots,m_{N_h}),\qquad
 \mathbf m_h=(m_1,\ldots,m_{N_h})^T,
\]
and define the cell-average projection
\begin{equation}
\label{eq:cell-average}
 (\mathcal P_hv)_i
 =\frac1{m_i}\int_{K_i}v(x)\,dx.
\end{equation}

For $i\ne j$, define the cell interaction and the corresponding rate by
\begin{equation}
\label{eq:cell-interactions}
 \alpha_{ij}^{\delta,h}
 =
 \int_{K_i}\int_{K_j}\kappa_\delta(x,y)\,dy\,dx,
 \qquad
 w_{ij}^{\delta,h}
 =\frac{\alpha_{ij}^{\delta,h}}{m_i}.
\end{equation}
The habitat-conforming discrete dispersal operator is
\begin{equation}
\label{eq:general-Dh}
 (D_{\delta,h}V)_i
 =
 \sum_{j\ne i}
 w_{ij}^{\delta,h}(V_j-V_i),
 \qquad i=1,\ldots,N_h.
\end{equation}

\begin{proposition}[Structural properties of the habitat-conforming operator]
\label{prop:habitat-structure}
For every $\delta>0$ and every partition \eqref{eq:mesh-partition}, the operator $D_{\delta,h}$ satisfies:
\begin{enumerate}[label=(\roman*)]
\item $(D_{\delta,h})_{ij}>0$ for $i\ne j$, and $D_{\delta,h}\one=0$;
\item
\begin{equation}
\label{eq:discrete-mass}
 \mathbf m_h^TD_{\delta,h}=0;
\end{equation}
\item
\begin{equation}
\label{eq:weighted-symmetry}
 M_hD_{\delta,h}=D_{\delta,h}^TM_h;
\end{equation}
\item for every $V\in\R^{N_h}$,
\begin{equation}
\label{eq:habitat-energy}
 -V^TM_hD_{\delta,h}V
 =
 \frac12\sum_{\substack{i,j=1\\ i\ne j}}^{N_h}
 \alpha_{ij}^{\delta,h}(V_i-V_j)^2
 \ge0;
\end{equation}
\item $\ker D_{\delta,h}=\operatorname{span}\{\one\}$.
\end{enumerate}
Consequently, the dispersal equation $V' = D_{\delta,h}V$ preserves the discrete total mass $\mathbf m_h^TV$.
\end{proposition}

\begin{proof}
The identities follow directly from the symmetric interaction coefficients.
Symmetry of $\kappa_\delta$ gives
\[
 \alpha_{ij}^{\delta,h}=\alpha_{ji}^{\delta,h},\qquad
 m_iw_{ij}^{\delta,h}=m_jw_{ji}^{\delta,h},\qquad
 M_hD_{\delta,h}=D_{\delta,h}^TM_h.
\]
Together with $D_{\delta,h}\one=0$, this yields $\mathbf m_h^TD_{\delta,h}=0$.
Pairing the $(i,j)$ and $(j,i)$ terms gives
\[
 -V^TM_hD_{\delta,h}V
 =
 \frac12\sum_{\substack{i,j=1\\ i\ne j}}^{N_h}
 \alpha_{ij}^{\delta,h}(V_i-V_j)^2
 \ge0.
\]
If $D_{\delta,h}V=0$, then the left-hand side vanishes.  Since $\kappa_\delta>0$ on $\overline\Omega\times\overline\Omega$ and each cell has positive measure, $\alpha_{ij}^{\delta,h}>0$ for $i\ne j$.  Therefore
\[
 \sum_{\substack{i,j=1\\ i\ne j}}^{N_h}
 \alpha_{ij}^{\delta,h}(V_i-V_j)^2=0
 \quad\Longrightarrow\quad
 V_i=V_j\quad\text{for all }i,j,
\]
so $\ker D_{\delta,h}=\operatorname{span}\{\one\}$.  Finally,
$\frac{d}{dt}(\mathbf m_h^TV)=\mathbf m_h^TD_{\delta,h}V=0$ along $V'=D_{\delta,h}V$.
\end{proof}

For $v\in C(\overline\Omega)$, define its modulus of continuity and the kernel bound by
\[
 \omega_v(\rho)=\sup\{|v(x)-v(y)|:x,y\in\overline\Omega,\ |x-y|\le\rho\},
 \qquad
 \Gamma_\delta=\sup_{x\in\overline\Omega}\int_\Omega\kappa_\delta(x,y)\,dy<\infty.
\]

\begin{proposition}[Consistency for fixed $\delta$]
\label{prop:general-fixed-delta-consistency}
For every fixed $\delta>0$ and every $v\in C(\overline\Omega)$,
\begin{equation}
\label{eq:general-fixed-scale}
 \norm{D_{\delta,h}\mathcal P_hv
       -\mathcal P_h\cD_\delta v}_\infty
 \le 2\Gamma_\delta\,\omega_v(h).
\end{equation}
In particular,
\[
 \norm{D_{\delta,h}\mathcal P_hv
       -\mathcal P_h\cD_\delta v}_\infty
 \longrightarrow0
 \qquad(h\to0).
\]
If $v$ is Lipschitz continuous, the right-hand side of \eqref{eq:general-fixed-scale} is $O(h)$.
\end{proposition}

\begin{proof}
Write $\bar v_i=(\mathcal P_hv)_i$.  From \eqref{eq:general-Dh} and \eqref{eq:cell-interactions},
\[
 (D_{\delta,h}\mathcal P_hv)_i
 =
 \frac1{m_i}
 \sum_{j=1}^{N_h}
 \int_{K_i}\int_{K_j}
 \kappa_\delta(x,y)(\bar v_j-\bar v_i)\,dy\,dx,
\]
where the $j=i$ term is zero.  On the other hand,
\[
 (\mathcal P_h\cD_\delta v)_i
 =
 \frac1{m_i}
 \sum_{j=1}^{N_h}
 \int_{K_i}\int_{K_j}
 \kappa_\delta(x,y)(v(y)-v(x))\,dy\,dx.
\]
For $x\in K_i$ and $y\in K_j$, $|\bar v_i-v(x)|\le\omega_v(h)$ and $|\bar v_j-v(y)|\le\omega_v(h)$.  Therefore
\[
 \left|
 (D_{\delta,h}\mathcal P_hv)_i
 -(\mathcal P_h\cD_\delta v)_i
 \right|
 \le
 2\omega_v(h)\,
 \frac1{m_i}
 \int_{K_i}\int_\Omega\kappa_\delta(x,y)\,dy\,dx
 \le2\Gamma_\delta\omega_v(h),
\]
which proves the estimate.
\end{proof}

\section{Semi-implicit two-stage scheme and threshold properties}
\label{sec:scheme}

For $q\in\{H,e,r,s,b,c,f,g\}$, write $q_i^h=(\mathcal P_hq)_i$.  We use
$H_h,E_h,R_h,S_h,\mathsf B_h,C_h,F_h,G_h$ for the corresponding diagonal
matrices of cell averages, reserving $B_h$ for the interstage coupling
matrix defined below.  In particular,
\[
\begin{aligned}
 H_h&=\diag(H_i^h),& E_h&=\diag(e_i^h),&
 R_h&=\diag(r_i^h),& S_h&=\diag(s_i^h),\\
 \mathsf B_h&=\diag(b_i^h),& C_h&=\diag(c_i^h),&
 F_h&=\diag(f_i^h),& G_h&=\diag(g_i^h).
\end{aligned}
\]

Define
\begin{equation}
\label{eq:Lh}
 L_{\delta,h}=A_{\delta,h}+B_h,
\end{equation}
where
\begin{equation}
\label{eq:split}
 A_{\delta,h}
 =\diag(\mu_1D_{\delta,h}-H_h,
         \mu_2D_{\delta,h}-E_h),
 \qquad
 B_h=
 \begin{pmatrix}
 0&R_h\\
 S_h&0
 \end{pmatrix}.
\end{equation}
By Proposition~\ref{prop:habitat-structure}, the off-diagonal coefficients
of the linear system generated by $L_{\delta,h}$ are nonnegative.  Moreover,
since $r,s\ge0$ and $r,s\not\equiv0$,
\[
 \sum_{i=1}^{N_h}m_i r_i^h
 =
 \int_\Omega r(x)\,dx>0,
 \qquad
 \sum_{i=1}^{N_h}m_i s_i^h
 =
 \int_\Omega s(x)\,dx>0,
\]
and therefore $R_h\ne0$ and $S_h\ne0$.

\begin{lemma}[Discrete maximum principle]
\label{lem:discrete-maximum}
Let $\mu>0$, let $Q=\diag(q_1,\ldots,q_{N_h})$ with $q_i\ge0$, and set $\mathcal A_h=\mu D_{\delta,h}-Q$.  For every $\Delta t>0$ and $F\ge0$, the equation $(I-\Delta t\mathcal A_h)V=F$ has a unique solution satisfying $V\ge0$; if $F\ne0$, then $V\gg0$.
\end{lemma}

\begin{corollary}[Strong positivity of the linearized evolution]
\label{cor:strong-positive-linear}
For $\mathcal A_h$ as in Lemma~\ref{lem:discrete-maximum},
\[
 V(0)\ge0,\quad V(0)\ne0
 \quad\Longrightarrow\quad
 e^{t\mathcal A_h}V(0)\gg0
 \qquad(t>0).
\]
Moreover,
\begin{equation}
\label{eq:strong-positive-linear}
 \U^0\ge0,\quad \U^0\ne0
 \quad\Longrightarrow\quad
 e^{tL_{\delta,h}}\U^0\gg0
 \qquad(t>0).
\end{equation}
\end{corollary}

Since $L_{\delta,h}\in\R^{2N_h\times2N_h}$, its spectrum is finite.
Hence its spectral bound is attained, and we define
\begin{equation}
\label{eq:lambda-h}
 \lambda_{\delta,h}
 :=\max\{\Re z:z\in\sigma(L_{\delta,h})\}.
\end{equation}
Thus $\max$ is used in the discrete problem, in contrast with the
$\sup$ in \eqref{eq:lambda-cont}.

Fix $t>0$ and set $T_h(t)=e^{tL_{\delta,h}}$.
By Corollary~\ref{cor:strong-positive-linear}, $T_h(t)$ is strongly
positive. The Perron--Frobenius theorem therefore gives a strictly positive eigenvector $\PhiV_h\gg0$ associated with the spectral radius $\rho(T_h(t))$, and this eigenvalue is simple; see, for example, \cite{Smith1995}. Since $T_h(t)=e^{tL_{\delta,h}}$ commutes with $L_{\delta,h}$,
\[
 T_h(t)L_{\delta,h}\PhiV_h
 =
 L_{\delta,h}T_h(t)\PhiV_h
 =
 \rho(T_h(t))L_{\delta,h}\PhiV_h.
\]
Thus $L_{\delta,h}\PhiV_h$ belongs to the eigenspace of $T_h(t)$
associated with $\rho(T_h(t))$.  Since this eigenspace is
one-dimensional,
\[
 L_{\delta,h}\PhiV_h=\mu_h\PhiV_h
\]
for some $\mu_h\in\R$.  By the finite-dimensional spectral mapping
theorem,
\[
 \rho(T_h(t))
 =
 e^{t\lambda_{\delta,h}}.
\]
On the other hand,
\[
 T_h(t)\PhiV_h
 =
 e^{t\mu_h}\PhiV_h
 =
 \rho(T_h(t))\PhiV_h,
\]
and hence $\mu_h=\lambda_{\delta,h}$.  Consequently,
\begin{equation}
\label{eq:discrete-principal-pair}
 L_{\delta,h}\PhiV_h
 =\lambda_{\delta,h}\PhiV_h,
 \qquad
 \PhiV_h\gg0,
\end{equation}
and $\lambda_{\delta,h}$ is a simple real eigenvalue.

For $\U=(U_1,U_2)^T\in\R_+^{2N_h}$, define
\begin{equation}
\label{eq:Q}
 Q_h(\U)
 =
 \diag\big(\mathsf B_hU_1+\tau C_hU_2,
            F_hU_2+\tau G_hU_1\big).
\end{equation}
The spatially discrete system is
\begin{equation}
\label{eq:semidiscrete}
 \U'(t)
 =
 \big(A_{\delta,h}+B_h-Q_h(\U(t))\big)\U(t).
\end{equation}
We use the semi-implicit Euler step
\begin{equation}
\label{eq:scheme}
 \big[I-\Delta t\big(A_{\delta,h}-Q_h(\U^n)\big)\big]\U^{n+1}
 =
 \big(I+\Delta t B_h\big)\U^n.
\end{equation}
Linearization at the zero population state gives
\begin{equation}
\label{eq:linearized-step}
 \U^{n+1}
 =
 G_{\delta,h}(\Delta t)\U^n,
 \qquad
 G_{\delta,h}(\Delta t)
 =
 \big(I-\Delta t A_{\delta,h}\big)^{-1}
 \big(I+\Delta t B_h\big).
\end{equation}

\subsection{Nonnegativity, fixed points, and the discrete threshold}

\begin{theorem}[Unconditional preservation of nonnegativity]
\label{thm:nonnegativity}
Under Assumption~\ref{ass:coeff}, for every $\Delta t>0$ and every $\U^n\ge0$, equation \eqref{eq:scheme} has a unique solution and $\U^{n+1}\ge0$.  If $\U^n\ne0$, then $U_1^{n+2}\gg0$ and $U_2^{n+2}\gg0$.  If $U_1^n\ne0$ and $U_2^n\ne0$, then already $U_1^{n+1}\gg0$ and $U_2^{n+1}\gg0$.
\end{theorem}

\begin{proposition}[Fixed points]
\label{prop:fixed-points}
A nonnegative vector $\U^*$ is a fixed point of \eqref{eq:scheme} if and only if it is a steady state of \eqref{eq:semidiscrete}:
\[
 \big(A_{\delta,h}+B_h-Q_h(\U^*)\big)\U^*=0.
\]
Thus the time step does not shift semidiscrete equilibria.
\end{proposition}

\begin{lemma}[Finite-time growth bound]
\label{lem:finite-bound}
Let $G_{\delta,h}(\Delta t)$ be defined by \eqref{eq:linearized-step}.  Then
\begin{equation}
\label{eq:comparison-G}
 0\le\U^{n+1}\le G_{\delta,h}(\Delta t)\U^n.
\end{equation}
Consequently, for $n\Delta t\le T$,
\[
 \norm{\U^n}_\infty
 \le C_T\norm{\U^0}_\infty,
\]
with $C_T$ independent of $h$, $\delta$, and $\Delta t$.
\end{lemma}

\begin{theorem}[Threshold equivalence]
\label{thm:threshold-equivalence}
For every $\Delta t>0$,
\begin{equation}
\label{eq:threshold-equivalence}
 \rho(G_{\delta,h})<1,\ =1,\ >1
 \quad\Longleftrightarrow\quad
 \lambda_{\delta,h}<0,\ =0,\ >0,
\end{equation}
respectively.  Equivalently,
\[
 \operatorname{sgn}\!\left(\rho(G_{\delta,h}(\Delta t))-1\right)
 =
 \operatorname{sgn}(\lambda_{\delta,h}).
\]
\end{theorem}

\begin{corollary}[Numerical extinction]
\label{cor:extinction}
If $\lambda_{\delta,h}<0$, then every solution of \eqref{eq:scheme} with $\U^0\ge0$ satisfies
\[
 \norm{\U^n}_\infty
 \le Cq^n\norm{\U^0}_\infty,
 \qquad 0<q<1.
\]
\end{corollary}

\begin{proposition}[Positive threshold]
\label{prop:positive-threshold}
Suppose $\lambda_{\delta,h}>0$.
\begin{enumerate}[label=(\roman*)]
\item For every $\Delta t>0$, $\rho(G_{\delta,h}(\Delta t))>1$, and hence $\U=0$ is a linearly unstable fixed point of \eqref{eq:scheme}.

\item If $\tau=0$, the semidiscrete system \eqref{eq:semidiscrete} admits a unique equilibrium $\U_h^*\gg0$, and
\[
 \lim_{t\to\infty}\U(t;\U^0)=\U_h^*
 \qquad\text{for every }\U^0\in\R_+^{2N_h}\setminus\{0\}.
\]

\item If $\U_h^*$ is hyperbolic and locally asymptotically stable for
\eqref{eq:semidiscrete}, then there exists $\Delta t_0>0$ such that,
for every $0<\Delta t<\Delta t_0$, $\U_h^*$ is a locally asymptotically
stable fixed point of \eqref{eq:scheme}.
\end{enumerate}
\end{proposition}

\subsection{Computable bounds for the continuous threshold}

\begin{lemma}[Bounds from an approximate eigenfunction]
\label{lem:residual-cont}
Let $\delta>0$ and let $\PhiV\in X_{++}\times X_{++}$ and $\zeta\in\R$.  Define
\begin{equation}
\label{eq:eta-cont}
 \eta(\PhiV,\zeta)
 =
 \max_{i=1,2}\sup_{x\in\overline\Omega}
 \frac{|(\cL_\delta\PhiV)_i(x)-\zeta\Phi_i(x)|}{\Phi_i(x)}.
\end{equation}
Then
\begin{equation}
\label{eq:residual-interval}
 \zeta-\eta(\PhiV,\zeta)
 \le\lambda_\delta
 \le\zeta+\eta(\PhiV,\zeta).
\end{equation}
\end{lemma}

\begin{corollary}[Sign determination from the bounds]
\label{cor:sign-interval}
Let $[\lambda_h^-,\lambda_h^+]$ denote bounds obtained from Lemma~\ref{lem:residual-cont}.  If $\lambda_h^+<0$, then $\lambda_\delta<0$.  If $\lambda_h^->0$, then $\lambda_\delta>0$.
\end{corollary}

When a positivity-preserving continuous reconstruction
$I_h\PhiV_h\in X_{++}\times X_{++}$ is available, substituting
$\PhiV=I_h\PhiV_h$ and $\zeta=\lambda_{\delta,h}$ into
\eqref{eq:eta-cont} gives an a posteriori interval for the continuous
principal spectrum point.

\section{Rectangular habitats: reflected Cartesian construction and local limit}
\label{sec:cartesian}
For rectangular habitats, we use a compactly supported reflected kernel and coordinatewise even reflection to recover the Neumann local limit.  In particular, $\mathcal D_\delta^{\Box}v\to\Delta v$ with homogeneous Neumann boundary condition and an $O(\delta^2)$ convergence rate.  Positivity of $J$ near the origin yields positive coordinate-neighbor interaction coefficients and the corresponding discrete maximum-principle and strict-positivity properties.  Nonlocal-to-local approximations with Neumann-type limits are studied in \cite{Cortazar2008,ShenXie2015}, and asymptotically compatible discretizations for parametrized nonlocal models are developed in \cite{TianDu2014,TianDu2020}.  In this section, let
\begin{equation}
\label{eq:box-domain}
 \Omega_{\Box}
 =
 \prod_{\ell=1}^{d}(0,L_{\ell}),
 \qquad L_\ell>0.
\end{equation}
Throughout the rectangular-habitat analysis we retain
Assumption~\ref{ass:coeff} and additionally assume $r(x)>0$ and $s(x)>0$ on $\overline{\Omega_{\Box}}$.

\subsection{Reflected nonlocal diffusion}

Let $B_1(0)$ denote the unit ball in $\R^d$.  Assume
\[
 J\in C_c(\R^d),\qquad
 \operatorname{supp}J\subset\overline{B_1(0)},
 \qquad
 J(z)=\widehat J(|z|)\ge0,
\]
with $J$ nonzero and positive in a neighborhood of the origin.  Fix
$\rho_J\in(0,1)$ such that
\begin{equation}
\label{eq:J-positive-ball}
 J(z)>0
 \qquad\text{for }|z|\le\rho_J.
\end{equation}
Define
\begin{equation}
\label{eq:m2}
 m_2
 =
 \int_{B_1(0)}J(z)z_1^2\,dz>0.
\end{equation}
Since $J(z)=J(|z|)$ and $B_1(0)$ is invariant under coordinate
reflections and permutations,
\begin{equation}
\label{eq:kernel-second-moment}
 \int_{B_1(0)}J(z)z_i z_j\,dz
 =
 \begin{cases}
  0, & i\ne j,\\
  m_2, & i=j,
 \end{cases}
 =
 m_2\delta_{ij}.
\end{equation}
For $v\in C(\overline{\Omega_{\Box}})$, let $E_Nv$ denote the repeated coordinatewise even reflection across the faces of $\Omega_{\Box}$.  Set $\delta_{\Box}:=\frac12\min_{1\le\ell\le d}L_\ell$.  For $0<\delta\le\delta_{\Box}$, define
\begin{equation}
\label{eq:nonlocal-op}
 (\cD_\delta^{\Box}v)(x)
 =
 \frac{2}{m_2\delta^2}
 \int_{B_1(0)}
 J(z)\big(E_Nv(x+\delta z)-v(x)\big)\,dz.
\end{equation}
At $\delta=0$, set
\begin{equation}
\label{eq:local-op}
 \cD_0^{\Box}v=\Delta v,
 \qquad
 \partial_\nu v=0\quad\text{on }\partial\Omega_{\Box}.
\end{equation}

\subsection{Cartesian discretization}

Choose $h>0$ such that $L_\ell=N_\ell h$ with $N_\ell\in\mathbb N$.  Let
\[
 \mathcal I_h
 =
 \prod_{\ell=1}^{d}\{0,\ldots,N_\ell-1\},
 \qquad
 N_h^{\Box}
 =
 \prod_{\ell=1}^{d}N_\ell,
\]
and use the cell centers
\[
 x_{\jj}
 =
 \big((j_1+\tfrac12)h,\ldots,(j_d+\tfrac12)h\big),
 \qquad
 \jj\in\mathcal I_h.
\]
Define $(\cR_hv)_{\jj}=v(x_{\jj})$.  For a grid function $V=(V_{\jj})_{\jj\in\mathcal I_h}$, let $V^E$ denote its repeated coordinatewise reflected extension to $\Z^d$.

For $\delta\ge h$, define
\[
 \mathcal Q_{\delta,h}
 =
 \{\qq\in\Z^d\setminus\{0\}:|h\qq|\le\delta\},
 \qquad
 \omega_{\qq}^{\delta,h}=J(h\qq/\delta),
\]
and
\begin{equation}
\label{eq:kappa}
 \kappa_{\delta,h}
 =
 \sum_{\qq\in\mathcal Q_{\delta,h}}
 \omega_{\qq}^{\delta,h}(hq_1)^2.
\end{equation}
Whenever $\kappa_{\delta,h}>0$, symmetry gives
\begin{equation}
\label{eq:disc-second-moment}
 \sum_{\qq\in\mathcal Q_{\delta,h}}
 \omega_{\qq}^{\delta,h}(hq_i)(hq_j)
 =
 \kappa_{\delta,h}\delta_{ij}.
\end{equation}
If $\delta\ge h$ and $h/\delta\le\rho_J$, define
\begin{equation}
\label{eq:Dh}
 (D_{\delta,h}^{\Box}V)_{\jj}
 =
 \frac{2}{\kappa_{\delta,h}}
 \sum_{\qq\in\mathcal Q_{\delta,h}}
 \omega_{\qq}^{\delta,h}
 \big(V^E_{\jj+\qq}-V_{\jj}\big).
\end{equation}
The cell-centered Neumann difference is
\begin{equation}
\label{eq:local-Dh}
 (D_{0,h}^{\Box}V)_{\jj}
 =
 \sum_{\ell=1}^{d}
 \frac{V^E_{\jj+\ee_{\ell}}-2V_{\jj}
       +V^E_{\jj-\ee_{\ell}}}{h^2}.
\end{equation}
For $0\le\delta<h$, set $D_{\delta,h}^{\Box}=D_{0,h}^{\Box}$.  For $\delta\ge h$, use \eqref{eq:Dh} when $h/\delta\le\rho_J$ and set $D_{\delta,h}^{\Box}=D_{0,h}^{\Box}$ otherwise.  If $h/\delta\le\rho_J$, then $\omega_{\pm\ee_\ell}^{\delta,h}=J((h/\delta)\ee_\ell)>0$ for $\ell=1,\ldots,d$, so in particular $\kappa_{\delta,h}>0$.  For every fixed $\delta>0$,
$h/\delta\to0$ as $h\to0$, and hence the nonlocal formula \eqref{eq:Dh}
is used for all sufficiently small $h$.  For $i\ne j$, write $w_{ij}^{\delta,h}=(D_{\delta,h}^{\Box})_{ij}$.

\begin{proposition}[Structural properties of the reflected Cartesian operator]
\label{prop:matrix-structure}
For $0\le\delta\le\delta_{\Box}$:
\begin{enumerate}[label=(\roman*)]
\item $(D_{\delta,h}^{\Box})_{ij}\ge0$ for $i\ne j$, and $D_{\delta,h}^{\Box}\one=0$;
\item $D_{\delta,h}^{\Box}$ is symmetric and negative semidefinite;
\item for every $V\in\R^{N_h^{\Box}}$,
\begin{equation}
\label{eq:energy}
 -h^dV^TD_{\delta,h}^{\Box}V
 =
 \frac{h^d}{2}
 \sum_{\substack{i,j=1\\ i\ne j}}^{N_h^{\Box}}
 w_{ij}^{\delta,h}(V_i-V_j)^2
 \ge0;
\end{equation}
\item $\ker D_{\delta,h}^{\Box}=\operatorname{span}\{\one\}$.
\end{enumerate}
\end{proposition}

\begin{proof}
Let $R:\mathbb Z^d\to\mathcal I_h$ denote the coordinatewise reflection map
defined by $V^E_{\kk}=V_{R(\kk)}$.  In one coordinate,
\[
 R_\ell^{-1}(k)
 =
 \{\,k+2mN_\ell,\,-k-1+2mN_\ell:m\in\mathbb Z\,\}.
\]
Thus, every reflected contribution from one cell to another has a
corresponding reverse contribution generated by an offset of the same Euclidean
length.  Since $J$ is radial, the two contributions have the same weight.
Consequently,
\[
w_{ij}^{\delta,h}=w_{ji}^{\delta,h}.
\]

For $i\ne j$, the coefficients in \eqref{eq:Dh} therefore satisfy
\[
 w_{ij}^{\delta,h}\ge0,\qquad w_{ij}^{\delta,h}=w_{ji}^{\delta,h},\qquad
 (D_{\delta,h}^{\Box})_{ii}=-\sum_{j\ne i}w_{ij}^{\delta,h}.
\]
Consequently,
\[
 (D_{\delta,h}^{\Box}\one)_i=-\sum_{j\ne i}w_{ij}^{\delta,h}+\sum_{j\ne i}w_{ij}^{\delta,h}=0,
 \qquad
 (D_{\delta,h}^{\Box})^T=D_{\delta,h}^{\Box}.
\]
Therefore
\[
 -h^dV^TD_{\delta,h}^{\Box}V
 =
 \frac{h^d}{2}
 \sum_{\substack{i,j=1\\ i\ne j}}^{N_h^{\Box}}
 w_{ij}^{\delta,h}(V_i-V_j)^2
 \ge0.
\]

If the nonlocal formula \eqref{eq:Dh} is used, then \eqref{eq:J-positive-ball} gives $\omega_{\pm\ee_\ell}^{\delta,h}>0$ for $\ell=1,\ldots,d$.  Thus equality in \eqref{eq:energy} implies equality of the values at every
pair of neighboring cell centers in each coordinate direction.  Repeating
these equalities along the Cartesian coordinate lines gives $V_{\jj}=V_{\kk}$ for all $\jj,\kk\in\mathcal I_h$.  If the local formula \eqref{eq:local-Dh} is used, then
\[
 -h^dV^TD_{0,h}^{\Box}V
 =
 h^{d-2}
 \sum_{\ell=1}^d
 \sum_{\substack{\jj\in\mathcal I_h\\0\le j_\ell\le N_\ell-2}}
 \big(V_{\jj+\ee_\ell}-V_{\jj}\big)^2.
\]
Hence
\[
 V^TD_{0,h}^{\Box}V=0
 \quad\Longrightarrow\quad
 V_{\jj+\ee_\ell}=V_{\jj}
\]
for every $\ell=1,\ldots,d$ and every
$\jj\in\mathcal I_h$ with $0\le j_\ell\le N_\ell-2$.  Successive coordinate steps therefore give $V_{\jj}=V_{\kk}$ for all $\jj,\kk\in\mathcal I_h$.  Thus in both cases $\ker D_{\delta,h}^{\Box}=\operatorname{span}\{\one\}$.
\end{proof}

\begin{proposition}[Fixed-$\delta$ consistency of the reflected operator]
\label{prop:fixed-delta-consistency}
Fix $0<\delta\le\delta_{\Box}$.  For every $v\in C(\overline{\Omega_{\Box}})$,
\begin{equation}
\label{eq:fixed-scale}
 \norm{D_{\delta,h}^{\Box}\cR_hv
       -\cR_h\cD_\delta^{\Box}v}_\infty
 \longrightarrow0
 \qquad(h\to0).
\end{equation}
\end{proposition}

\begin{proof}
For fixed $\delta>0$, \eqref{eq:Dh} is used for all sufficiently small $h$.  Set $\theta_h=h/\delta$.  Then
\[
 (D_{\delta,h}^{\Box}\cR_hv)_{\jj}
 =
 \frac{2}{\delta^2}
 \frac{\displaystyle
 \theta_h^d\sum_{\qq\in\Z^d}
 J(\theta_h\qq)
 \big(E_Nv(x_{\jj}+\delta\theta_h\qq)-v(x_{\jj})\big)}
 {\displaystyle
 \theta_h^d\sum_{\qq\in\Z^d}
 J(\theta_h\qq)(\theta_hq_1)^2}.
\]
The denominator converges to $m_2$.  Define $F(x,z):=J(z)(E_Nv(x+\delta z)-v(x))$.
Since $F$ is continuous on
$\overline{\Omega_{\Box}}\times\overline{B_1(0)}$, it is uniformly
continuous.  Hence
\[
 \sup_{x\in\overline{\Omega_{\Box}}}
 \left|
 \theta_h^d\sum_{\qq\in\Z^d}F(x,\theta_h\qq)
 -
 \int_{B_1(0)}F(x,z)\,dz
 \right|
 \longrightarrow0
 \qquad(h\to0).
\]
Division by the positive limit $m_2$ proves \eqref{eq:fixed-scale}.
\end{proof}

\subsection{Local-limit consistency}

For a scalar function $w\in C^4(\R^d)$, set $|D^4w(x)|=\sup_{|\xi|=1}|D^4w(x)[\xi,\xi,\xi,\xi]|$. The regularity condition $E_Nv\in C^4(\R^d)$ has a simple
facewise characterization that will be useful below.  If
$v\in C^4(\overline{\Omega_{\Box}})$, then its repeated coordinatewise even
extension is $C^4$ precisely when, for every coordinate $\ell$,
\begin{equation}
\label{eq:reflection-compatibility}
 \partial_\ell v=0,
 \qquad
 \partial_\ell^3 v=0
 \quad\text{on the two faces }x_\ell=0\text{ and }x_\ell=L_\ell.
\end{equation}
For an even reflection across the face
$x_\ell=0$ or $x_\ell=L_\ell$, derivatives involving an odd number of
differentiations with respect to $x_\ell$ change sign across the face, whereas
those involving an even number do not.  Consequently, $C^4$-regularity across
the face requires
\[
\partial_\ell v=\partial_\ell^3v=0.
\]
Mixed derivatives involving an odd number of differentiations with respect to
$x_\ell$ satisfy the corresponding compatibility conditions by tangentially
differentiating $\partial_\ell v=0$ or $\partial_\ell^3v=0$, as appropriate.
Applying this argument successively to each coordinate gives
$C^4$-compatibility at edges and corners as well.

\begin{lemma}[Continuous local consistency on the box]
\label{lem:cont-local-consistency}
Suppose $v\in C^4(\overline{\Omega_{\Box}})$ and $E_Nv\in C^4(\R^d)$.  Let $m_4=\int_{B_1(0)}J(z)|z|^4\,dz$.  Then, for $0<\delta\le\delta_{\Box}$,
\begin{equation}
\label{eq:cont-local-error}
 \norm{\cD_\delta^{\Box}v-\Delta v}_{L^\infty(\Omega_{\Box})}
 \le
 \frac{m_4}{12m_2}\,\delta^2
 \norm{|D^4E_Nv|}_{L^\infty(\R^d)}.
\end{equation}
\end{lemma}

\begin{proof}
Taylor expansion of $E_Nv(x+\delta z)$ about $x$ gives
\[
 E_Nv(x+\delta z)-v(x)
 =
 \delta\nabla v(x)\cdot z
 +\frac{\delta^2}{2}D^2v(x)[z,z]
 +\frac{\delta^3}{6}D^3v(x)[z,z,z]
 +R_4(x,z),
\]
with
\[
 |R_4(x,z)|
 \le
 \frac{\delta^4|z|^4}{24}
 \norm{|D^4E_Nv|}_{L^\infty(\R^d)}.
\]
The odd terms integrate to zero under $z\mapsto-z$.  By
\eqref{eq:kernel-second-moment},
\[
 \int_{B_1(0)}J(z)D^2v(x)[z,z]\,dz
 =
 \sum_{i,j=1}^d
 \partial_{ij}v(x)
 \int_{B_1(0)}J(z)z_i z_j\,dz
 =
 m_2\sum_{i=1}^d\partial_{ii}v(x)
 =
 m_2\Delta v(x).
\]
Multiplication by $2/(m_2\delta^2)$ proves \eqref{eq:cont-local-error}.
\end{proof}

\begin{theorem}[Discrete local-limit consistency]
\label{thm:discrete-local-consistency}
Under the assumptions of Lemma~\ref{lem:cont-local-consistency},
\begin{equation}
\label{eq:discrete-local-consistency}
 \norm{D_{\delta,h}^{\Box}\cR_hv-\cR_h\Delta v}_\infty
 \le C_v(\delta^2+h^2),
 \qquad
 0\le\delta\le\delta_{\Box},
\end{equation}
where $C_v$ is independent of the ratio $h/\delta$.
\end{theorem}

\begin{proof}
If \eqref{eq:local-Dh} is used, then Taylor expansion of the
$C^4$ even extension gives
\[
 \norm{D_{0,h}^{\Box}\cR_hv-\cR_h\Delta v}_\infty
 \le
 \frac{h^2}{12}
 \sum_{\ell=1}^d
 \norm{\partial_\ell^4E_Nv}_{L^\infty(\R^d)}.
\]
Otherwise $\delta\ge h$, $h/\delta\le\rho_J$, and \eqref{eq:Dh} is used.  Taylor expansion gives
\[
\begin{aligned}
 E_Nv(x_{\jj}+h\qq)-v(x_{\jj})
 &=
 h\nabla v(x_{\jj})\cdot\qq
 +\frac{h^2}{2}D^2v(x_{\jj})[\qq,\qq]\\
 &\quad
 +\frac{h^3}{6}D^3v(x_{\jj})[\qq,\qq,\qq]
 +R_{\jj,\qq}.
\end{aligned}
\]
The odd terms cancel under $\qq\leftrightarrow-\qq$, and
\eqref{eq:disc-second-moment} makes the quadratic term exactly
$\Delta v(x_{\jj})$.  For the fourth-order remainder,
\[
 |R_{\jj,\qq}|
 \le
 \frac{|h\qq|^4}{24}
 \norm{|D^4E_Nv|}_{L^\infty(\R^d)}.
\]
Since $|h\qq|\le\delta$ for every contributing offset,
\[
 \sum_{\qq}\omega_{\qq}^{\delta,h}|h\qq|^4
 \le
 \delta^2
 \sum_{\qq}\omega_{\qq}^{\delta,h}|h\qq|^2.
\]
Moreover, summing \eqref{eq:disc-second-moment} over the coordinate directions gives
\[
 \sum_{\qq}\omega_{\qq}^{\delta,h}|h\qq|^2
 =d\,\kappa_{\delta,h}.
\]
Consequently,
\[
 \left|
 \frac{2}{\kappa_{\delta,h}}
 \sum_{\qq}\omega_{\qq}^{\delta,h}R_{\jj,\qq}
 \right|
 \le
 \frac{d}{12}\,\delta^2
 \norm{|D^4E_Nv|}_{L^\infty(\R^d)}.
\]
The factor $\kappa_{\delta,h}$ cancels explicitly, so the constant in this
estimate does not depend on $h/\delta$.  Combining the two cases proves
\eqref{eq:discrete-local-consistency}.
\end{proof}

\subsection{Box thresholds and asymptotic compatibility}

Define
\[
 \cL_\delta^{\Box}
 =
 \begin{pmatrix}
 \mu_1\cD_\delta^{\Box}-H&r\\
 s&\mu_2\cD_\delta^{\Box}-e
 \end{pmatrix},
 \qquad
 \lambda_\delta^{\Box}
 :=
 \sup\{\Re z:z\in\sigma(\cL_\delta^{\Box})\},
\]
and let $\lambda_0^{\Box}$ be the principal eigenvalue of the corresponding local Neumann operator $\cL_0^{\Box}$.

Let $A_{\delta,h}^{\Box}$, $B_h^{\Box}$, $Q_h^{\Box}$, and $L_{\delta,h}^{\Box}$ be obtained from \eqref{eq:split}, \eqref{eq:Q}, and \eqref{eq:Lh} by replacing $D_{\delta,h}$ with $D_{\delta,h}^{\Box}$ and evaluating the coefficients at the cell centers.  Since $L_{\delta,h}^{\Box}$ is finite dimensional, its spectral bound is attained; define
\[
 \lambda_{\delta,h}^{\Box}
 :=
 \max\{\Re z:z\in\sigma(L_{\delta,h}^{\Box})\},
\]
and denote the case $\delta=0$ by $\lambda_{0,h}^{\Box}$.  The corresponding semi-implicit solution is denoted by $\U_{\Box}^n$.

We give the corresponding positivity property for the reflected Cartesian
operator.  Let
$Q=\diag(q_{\jj})$ with $q_{\jj}\ge0$, let $\mu>0$, and set
$\mathcal A_h^{\Box}=\mu D_{\delta,h}^{\Box}-Q$.  For every $\Delta t>0$ and
$F\ge0$, the equation
\[
 (I-\Delta t\mathcal A_h^{\Box})V=F
\]
has a unique solution $V\ge0$.  The proof uses the same minimum-value argument as in Lemma~\ref{lem:discrete-maximum}, since all off-diagonal coefficients of
$D_{\delta,h}^{\Box}$ are nonnegative.  If $F\ne0$ and some component
$V_{\jj}=0$, then
\[
 0\le F_{\jj}
 =-\Delta t\,\mu
   \sum_{\kk\ne\jj}w_{\jj\kk}^{\delta,h}V_{\kk}
 \le0.
\]
Hence every cell coupled to $\jj$ by a positive coefficient also has value
zero.  When the nonlocal formula \eqref{eq:Dh} is used,
\eqref{eq:J-positive-ball} gives positive coefficients between successive cell
centers in each coordinate direction; the local formula \eqref{eq:local-Dh}
has the same nearest-neighbor property.  Repeating the preceding implication coordinate by coordinate reaches every
cell center and therefore gives $V\equiv0$.  Substitution into
$(I-\Delta t\mathcal A_h^{\Box})V=F$ would then give $F=0$, contradicting
$F\ne0$.  Hence $V\gg0$.

The same successive coordinate argument applied to the variation-of-constants
formula shows that
\[
 V(0)\ge0,\quad V(0)\ne0
 \quad\Longrightarrow\quad
 e^{t\mathcal A_h^{\Box}}V(0)\gg0
 \qquad(t>0).
\]
Therefore the scalar maximum-principle and strict-positivity steps used in
Section~\ref{sec:scheme} hold for the reflected Cartesian operator by a direct
coordinatewise proof.  Since $r$ and $s$ are strictly positive on
$\overline{\Omega_{\Box}}$, their values at every cell center are positive.
Hence
\[
R_h^{\Box}=\diag\big(r(x_{\jj})\big)_{\jj\in\mathcal I_h},
\qquad
S_h^{\Box}=\diag\big(s(x_{\jj})\big)_{\jj\in\mathcal I_h}
\]
have strictly positive diagonal entries, so positivity transfers between the
two stages exactly as in Corollary~\ref{cor:strong-positive-linear}.  Consequently the nonnegativity,
strong-positivity, principal-pair, and threshold arguments of
Section~\ref{sec:scheme} apply to the reflected Cartesian scheme for every
admissible $h$ and $\delta$.

\begin{assumption}[Smooth positive local eigenfunction on the box]
\label{ass:eigenfunction}
The local problem has a normalized eigenfunction $\PhiV^{0,\Box}\in C^4(\overline{\Omega_{\Box}})^2$ satisfying
\[
 \cL_0^{\Box}\PhiV^{0,\Box}=\lambda_0^{\Box}\PhiV^{0,\Box},
 \qquad
 \min_{x\in\overline{\Omega_{\Box}},\,i}\Phi_i^{0,\Box}(x)=m_\Phi>0,
\]
and each component has a coordinatewise even extension in $C^4(\R^d)$; equivalently, each component satisfies \eqref{eq:reflection-compatibility} on every face.
\end{assumption}

A useful sufficient condition for the third-order part of
Assumption~\ref{ass:eigenfunction} follows directly from the local eigenvalue
system.  Suppose, in addition to the Neumann condition, that
\[
 \partial_\ell H=\partial_\ell r=\partial_\ell s=\partial_\ell e=0
 \quad\text{on }x_\ell=0,L_\ell
\]
for every $\ell$.  Differentiate the two local eigenvalue equations in
$x_\ell$.  Since $\partial_\ell\Phi_i^{0,\Box}=0$ along the whole face, all of
its tangential derivatives vanish there.  Consequently
$\partial_\ell\Delta\Phi_i^{0,\Box}=\partial_\ell^3\Phi_i^{0,\Box}$ on that
face, and the differentiated equations give
$\partial_\ell^3\Phi_i^{0,\Box}=0$.  Thus reflection-symmetric coefficients
provide a directly verifiable route to the $C^4$ even-extension hypothesis.

Let $\uvec^{0,\Box}$ denote the classical solution of the local Neumann two-stage system on $\Omega_{\Box}$.

\begin{assumption}[Regular local solution on the box]
\label{ass:solution-reg}
For a fixed $T>0$,
\[
 u_i^{0,\Box}
 \in C^2([0,T];C(\overline{\Omega_{\Box}}))
 \cap C^1([0,T];C^4(\overline{\Omega_{\Box}})),
\]
and $E_Nu_i^{0,\Box}(t,\cdot)\in C^4(\R^d)$ uniformly for $0\le t\le T$; equivalently, \eqref{eq:reflection-compatibility} holds for $u_i^{0,\Box}(t,\cdot)$ on every face, uniformly in $t$.  The solution remains uniformly bounded and nonnegative on $[0,T]\times\overline{\Omega_{\Box}}$.
\end{assumption}

\begin{theorem}[Continuous threshold convergence on the box]
\label{thm:continuous-spectral-ac}
Under Assumptions~\ref{ass:coeff} and \ref{ass:eigenfunction},
\begin{equation}
\label{eq:continuous-spectral-ac}
 |\lambda_\delta^{\Box}-\lambda_0^{\Box}|
 \le C\delta^2.
\end{equation}
\end{theorem}

\begin{theorem}[Discrete threshold convergence on the box]
\label{thm:discrete-spectral-ac}
Under the same assumptions,
\begin{equation}
\label{eq:spectral-ac}
 |\lambda_{\delta,h}^{\Box}-\lambda_0^{\Box}|
 \le C(\delta^2+h^2),
\end{equation}
where $C$ is independent of $h/\delta$.
\end{theorem}

\begin{corollary}[Approximation of the reflected continuous threshold]
\label{cor:nonlocal-threshold-approx}
For all sufficiently small $h$ and $\delta$,
\[
 |\lambda_{\delta,h}^{\Box}-\lambda_\delta^{\Box}|
 \le C(h^2+\delta^2),
\]
with $C$ independent of $h/\delta$.
\end{corollary}

\begin{corollary}[Correct local threshold sign]
\label{cor:eventual-sign}
If $\lambda_0^{\Box}\ne0$, then $\operatorname{sign}\lambda_{\delta,h}^{\Box}=\operatorname{sign}\lambda_0^{\Box}$ for all sufficiently small $h$ and $\delta$, independently of their relative rate of decay.
\end{corollary}

\begin{theorem}[Finite-time convergence to the local box solution]
\label{thm:nonlinear-ac}
Let $\U_{\Box}^0=\cR_h\uvec^{0,\Box}(0)$.  Under Assumptions~\ref{ass:coeff} and \ref{ass:solution-reg}, there exist $h_\star,\delta_\star,\Delta t_\star>0$ and $C_T>0$ such that
\begin{equation}
\label{eq:nonlinear-error}
 \max_{0\le n\Delta t\le T}
 \norm{\U_{\Box}^n-\cR_h\uvec^{0,\Box}(t_n)}_\infty
 \le
 C_T(\Delta t+h^2+\delta^2).
\end{equation}
The constant is independent of the ratio $h/\delta$.
\end{theorem}

\section{Numerical experiments}
\label{sec:numerics}

\subsection{Smooth nonrectangular habitat}

To illustrate the habitat-conforming construction, consider the smooth elliptical habitat
\[
 \Omega_{\mathrm{ell}}
 =
 \left\{(x,y)\in\R^2:
 \frac{x^2}{1.2^2}+\frac{y^2}{0.75^2}<1\right\}.
\]
We divide $0\le r\le1$ into six radial intervals and $0\le\theta<2\pi$ into twenty-four equal angular intervals, yielding $144$ cells after $(r,\theta)\mapsto(1.2r\cos\theta,0.75r\sin\theta)$.  Let
\[
 r_k=\frac{k}{6},\qquad \theta_\ell=\frac{2\pi(\ell+\tfrac12)}{24},\qquad
 \bar r_k=\frac{2}{3}\frac{r_{k+1}^3-r_k^3}{r_{k+1}^2-r_k^2}.
\]
Then the representative point and cell area are
\[
 x_{k,\ell}=(1.2\bar r_k\cos\theta_\ell,0.75\bar r_k\sin\theta_\ell),
 \qquad
 m_{k,\ell}=(1.2)(0.75)\frac{\pi}{24}(r_{k+1}^2-r_k^2).
\]
Thus $N_h=144$.  For each pair of cells $K_i$ and $K_j$, we approximate
\[
 \int_{K_i}\int_{K_j}\kappa_\delta(x,y)\,dy\,dx
 \approx
 |K_i||K_j|\,\kappa_\delta(x_i,x_j)
 =
 m_im_j\kappa_\delta(x_i,x_j),
\]
where $x_i\in K_i$ and $x_j\in K_j$ are the representative points.  Since $\kappa_\delta(x_i,x_j)=\kappa_\delta(x_j,x_i)$, the approximation satisfies $\alpha_{ij}^{\delta,h}=\alpha_{ji}^{\delta,h}$ up to roundoff.  Consequently,
\[
 m_iw_{ij}^{\delta,h}=m_jw_{ji}^{\delta,h},
 \qquad
 D_{\delta,h}\one=0,
 \qquad
 \mathbf m_h^TD_{\delta,h}=0
\]
up to roundoff.

We use
\[
 \kappa_\delta(x,y)
 =
 \frac1{\pi\delta^2}
 \exp\!\left(-\frac{|x-y|^2}{\delta^2}\right),
 \qquad
 \delta=0.35,
\]
and
\[
\begin{aligned}
 a(x,y)&=0.35+0.04x,
 &s(x,y)&=1.10+0.18y^2,\\
 r_\gamma(x,y)&=\gamma(1.40-0.18x),
 &e(x,y)&=0.72+0.06y,
\end{aligned}
\]
with $\mu_1=0.4$ and $\mu_2=1$.  For this experiment, the coefficient cell averages are evaluated by the same representative-point rule.  In double precision the computed structural residuals are
\[
 \norm{D_{\delta,h}\one}_\infty
 =5.6\times10^{-16},
 \qquad
 \norm{\mathbf m_h^TD_{\delta,h}}_\infty
 =6.9\times10^{-18}.
\]
Table~\ref{tab:ellipse-threshold} shows the discrete principal eigenvalue as the reproduction coefficient is scaled by $\gamma$.  The sign change occurs without any rectangular geometry or even-reflection construction.

\begin{table}[htbp]
\centering
\caption{Discrete threshold on the smooth elliptical habitat.}
\label{tab:ellipse-threshold}
\begin{tabular}{cc}
\toprule
$\gamma$ & $\lambda_{\delta,h}$\\
\midrule
0.45 & $-0.152231$\\
0.65 & $\phantom{-}0.012332$\\
1.00 & $\phantom{-}0.253676$\\
\bottomrule
\end{tabular}
\end{table}

\subsection{Rectangular habitat and local-limit convergence}

For the reflected Cartesian experiments, take the one-dimensional rectangular habitat
\[
 \Omega_{\Box}=(0,1),
\]
the compact kernel $J(z)=(1-z^2)_+$, and cell-centered reflection.  The heterogeneous coefficients are
\[
\begin{aligned}
 a(x)&=0.35+0.05\cos(2\pi x),
 &s(x)&=1.10+0.25\cos(2\pi x),\\
 r(x)&=1.55-0.50\cos(2\pi x),
 &e(x)&=0.72+0.10\cos(2\pi x),
\end{aligned}
\]
with $\mu_1=0.4$ and $\mu_2=1$.  Unless stated otherwise,
\[
 b(x)=0.8+0.1\cos(2\pi x),\qquad
 f(x)=0.7+0.1\sin^2(2\pi x),
\]
and $c=0.25$, $g=0.20$.  For the time-dependent experiments we take $\tau=0$ and use
\[
 u_1^0(x)=0.15+0.05\cos(2\pi x),\qquad
 u_2^0(x)=0.10+0.03\sin^2(2\pi x).
\]
These coefficient and initial-data profiles extend smoothly by even reflection
across $x=0$ and $x=1$.  Their first and third derivatives vanish at both
endpoints, so the prescribed profiles satisfy
\eqref{eq:reflection-compatibility}.  The coefficient profiles also satisfy
the sufficient reflection-symmetry condition stated after
Assumption~\ref{ass:eigenfunction}.

Figure~\ref{fig:spectral-ac} compares the reflected discrete nonlocal threshold with the discrete local threshold on a fixed fine grid.  The observed log--log slope is approximately $2.08$, consistent with the predicted second-order local-limit behavior on a sufficiently fine fixed mesh.  The plotted error measures convergence of $\lambda_{\delta,h}^{\Box}$ to the discrete local threshold $\lambda_{0,h}^{\Box}$ on the fixed grid.

\begin{figure}[htbp]
\centering
\includegraphics[width=.64\textwidth]{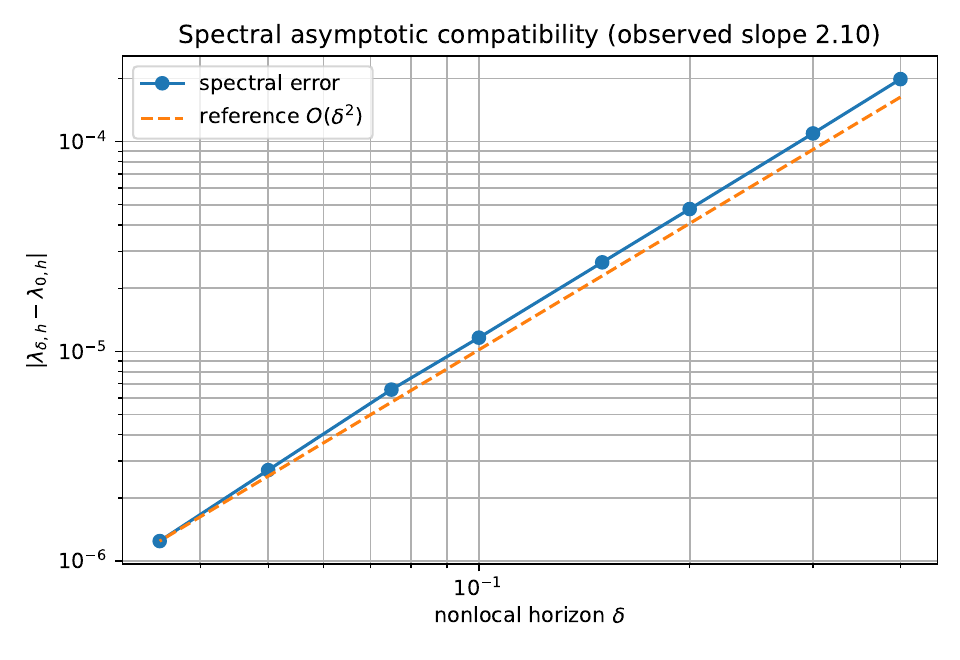}
\caption{Convergence of the reflected nonlocal discrete threshold to the local discrete threshold as $\delta\downarrow0$.  The mesh has $N=140$ cell centers.}
\label{fig:spectral-ac}
\end{figure}

\begin{table}[htbp]
\centering
\caption{Representative convergence data for the reflected Cartesian threshold.}
\label{tab:spectral}
\begin{tabular}{ccc}
\toprule
$\delta$ & $\lambda_{\delta,h}^{\Box}$ &
$|\lambda_{\delta,h}^{\Box}-\lambda_{0,h}^{\Box}|$\\
\midrule
0.400 & 0.269547 & $1.98\times10^{-4}$\\
0.200 & 0.269698 & $4.77\times10^{-5}$\\
0.100 & 0.269734 & $1.16\times10^{-5}$\\
0.050 & 0.269743 & $2.72\times10^{-6}$\\
0.035 & 0.269745 & $1.25\times10^{-6}$\\
\bottomrule
\end{tabular}
\end{table}

\subsection{Rectangular threshold dynamics and nonnegativity}

We use $N=80$, $\delta=0.15$, $\Delta t=0.02$, and integrate to $t=30$.
Multiplying $r$ by $0.45$ gives $\lambda_{\delta,h}^{\Box}=-0.1361$; the mean
total density decreases to $2.57\times10^{-3}$ at $t=30$.  With the unscaled
$r$, $\lambda_{\delta,h}^{\Box}=0.2697$, and the computed trajectory approaches
a positive fixed point.  By Proposition~\ref{prop:fixed-points}, such a fixed
point is a semidiscrete equilibrium, while Proposition~\ref{prop:positive-threshold}
gives uniqueness of the positive semidiscrete equilibrium when $\tau=0$.
All computed components remain nonnegative throughout both runs.

\begin{figure}[htbp]
\centering
\includegraphics[width=.62\textwidth]{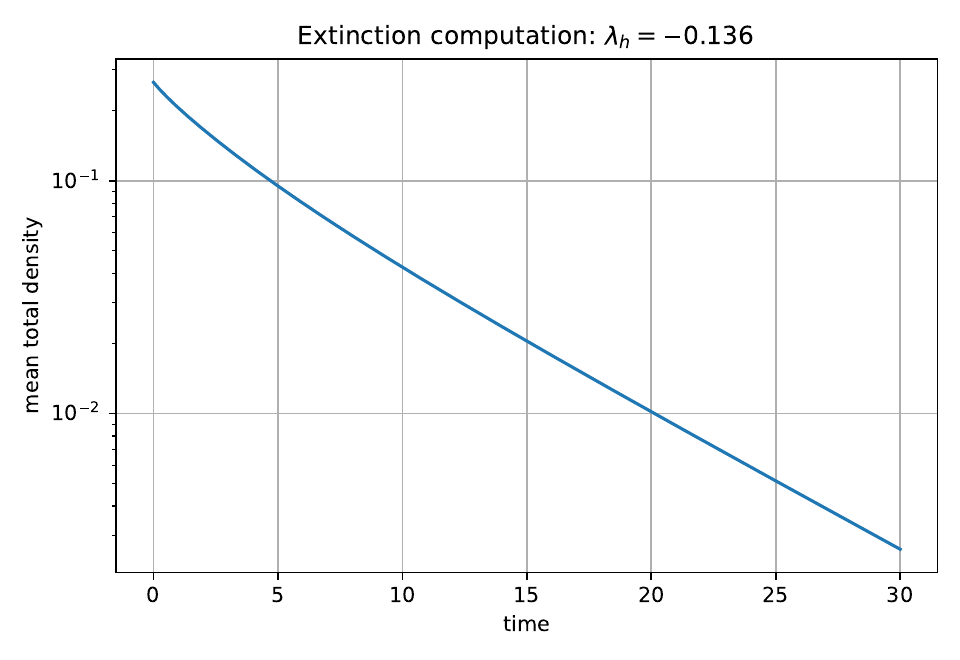}
\caption{Extinction computation on the rectangular habitat when the reflected discrete threshold is negative.}
\label{fig:extinction}
\end{figure}

\begin{figure}[htbp]
\centering
\includegraphics[width=.62\textwidth]{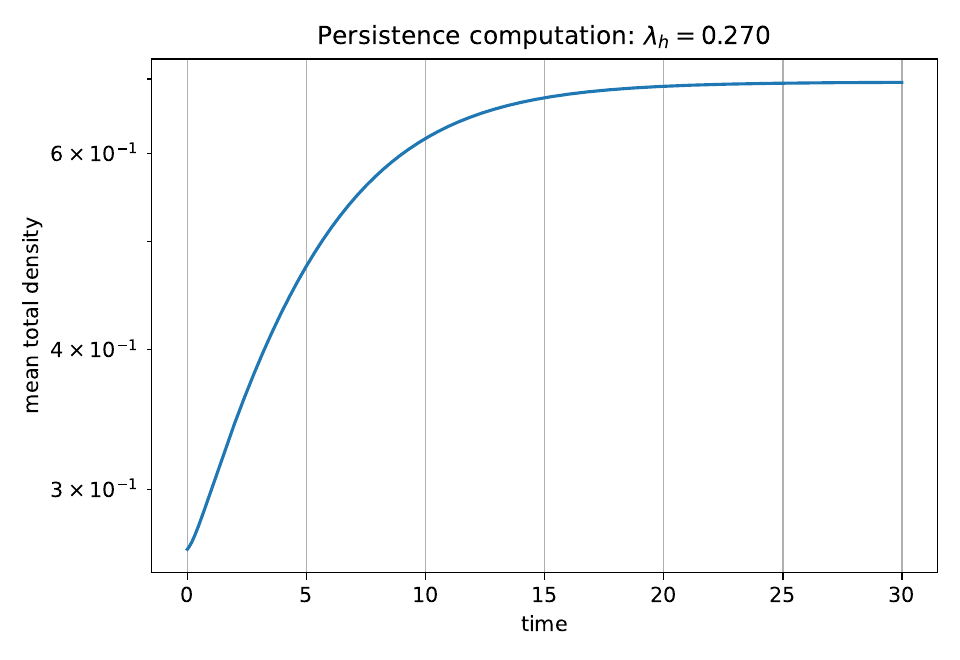}
\caption{Positive long-time profile on the rectangular habitat when the reflected discrete threshold is positive and $\tau=0$.}
\label{fig:persistence}
\end{figure}

\section{Proofs of the threshold and convergence results}
\label{sec:proof-main}

\subsection{Habitat-conforming scheme}

\begin{proof}[Proof of Lemma~\ref{lem:discrete-maximum}]
We use a minimum-value argument; compare \cite{EymardHandlovicovaMikula2011}.  Set $C_{\Delta t}=I-\Delta t\mathcal A_h=I-\Delta t(\mu D_{\delta,h}-Q)$.  Suppose $C_{\Delta t}V=F\ge0$.  If $V_k=\min_{1\le i\le N_h}V_i<0$, then
\[
 (D_{\delta,h}V)_k
 =\sum_{j\ne k}w_{kj}^{\delta,h}(V_j-V_k)\ge0,
\]
and therefore
\[
 F_k
 =(1+\Delta t q_k)V_k
 -\Delta t\mu(D_{\delta,h}V)_k<0,
\]
a contradiction.  Hence $V\ge0$.  If $F=0$, applying the same argument
to $-V$ gives $V\le0$, and therefore $V=0$.  Thus $\ker C_{\Delta t}=\{0\}$.  Since $C_{\Delta t}\in\R^{N_h\times N_h}$, rank--nullity gives $\operatorname{rank}C_{\Delta t}=N_h-\dim\ker C_{\Delta t}=N_h$, and hence $C_{\Delta t}$ is invertible.

If $F\ne0$ and $V_k=0$ for some $k$, then $V_k=\min_iV_i$ and
\[
 0\le F_k
 =-\Delta t\mu\sum_{j\ne k}w_{kj}^{\delta,h}V_j\le0.
\]
Since $w_{kj}^{\delta,h}>0$ for every $j\ne k$, it follows that
$V_j=0$ for all $j$, hence $V=0$ and therefore $F=0$, a contradiction.
Thus $V\gg0$ whenever $F\ne0$.
\end{proof}

\begin{proof}[Proof of Corollary~\ref{cor:strong-positive-linear}]
Each component of $V'=\mathcal A_hV$ satisfies
\[
 V_i'(t)
 +\left(q_i+\mu\sum_{j\ne i}w_{ij}^{\delta,h}\right)V_i(t)
 =\mu\sum_{j\ne i}w_{ij}^{\delta,h}V_j(t).
\]
Set $a_i=q_i+\mu\sum_{j\ne i}w_{ij}^{\delta,h}$.  Then
\[
 V_i(t)
 =
 e^{-a_it}V_i(0)
 +
 \mu\int_0^t
 e^{-a_i(t-s)}
 \sum_{j\ne i}w_{ij}^{\delta,h}V_j(s)\,ds.
\]
Hence $V(0)\ge0$ implies $V(t)\ge0$ for $t\ge0$.  If $V_k(0)>0$ for some $k$, then $V_k(t)\ge e^{-a_kt}V_k(0)>0$, and, for $i\ne k$,
\[
 V_i(t)
 \ge
 \mu w_{ik}^{\delta,h}
 \int_0^t e^{-a_i(t-s)}V_k(s)\,ds
 >0.
\]
Thus $V(t)\gg0$ for every $t>0$.
For the two-stage system, this argument first spreads positivity within
any nonzero stage.  Since $R_h\ne0$ and $S_h\ne0$, a positive stage then
provides a nonzero nonnegative forcing to the other stage, where the same
argument spreads positivity throughout the habitat.  This proves
\eqref{eq:strong-positive-linear}.
\end{proof}

\begin{proof}[Proof of Theorem~\ref{thm:nonnegativity}]
The right-hand side of \eqref{eq:scheme} satisfies $F^n:=(I+\Delta t B_h)\U^n\ge0$.
The left-hand side is block diagonal, and each diagonal block has the form
covered by Lemma~\ref{lem:discrete-maximum}, with a nonnegative diagonal
coefficient.  Hence \eqref{eq:scheme} has a unique solution and
$\U^{n+1}\ge0$.

If, for example, $U_1^n\ne0$, then the first-stage component of $F^n$ is
nonzero, so Lemma~\ref{lem:discrete-maximum} gives $U_1^{n+1}\gg0$.
Because $S_h\ne0$, $S_hU_1^{n+1}\ne0$, and the second-stage right-hand side at the next step is nonzero; hence
$U_2^{n+2}\gg0$.  Since $U_1^{n+1}\gg0$, the first-stage right-hand side
at the next step is also nonzero, and therefore $U_1^{n+2}\gg0$.  The
argument is identical when the initially nonzero component is $U_2^n$.
Thus $\U^n\ne0$ implies $U_1^{n+2}\gg0$ and $U_2^{n+2}\gg0$.  If both $U_1^n\ne0$ and $U_2^n\ne0$, then both stage components of $F^n$ are nonzero, so Lemma~\ref{lem:discrete-maximum} gives $U_1^{n+1}\gg0$ and $U_2^{n+1}\gg0$.
\end{proof}

\begin{proof}[Proof of Proposition~\ref{prop:fixed-points}]
If $\U^{n+1}=\U^n=\U^*$ in \eqref{eq:scheme}, cancellation of $\U^*$ and division by $\Delta t$ give
\[
 \big(A_{\delta,h}+B_h-Q_h(\U^*)\big)\U^*=0.
\]
The converse follows by reversing the calculation.
\end{proof}

\begin{proof}[Proof of Lemma~\ref{lem:finite-bound}]
Set
\[
 C_0=I-\Delta tA_{\delta,h},
 \qquad
 C_n=I-\Delta t\big(A_{\delta,h}-Q_h(\U^n)\big),
\]
and let
\[
 F^n=(I+\Delta tB_h)\U^n\ge0.
\]
Then
\[
 C_n\U^{n+1}=F^n,
 \qquad
 C_0G_{\delta,h}\U^n=F^n.
\]
By Theorem~\ref{thm:nonnegativity}, $\U^{n+1}\ge0$, and
\[
 C_0\big(G_{\delta,h}\U^n-\U^{n+1}\big)
 =\Delta t\,Q_h(\U^n)\U^{n+1}\ge0.
\]
The discrete maximum principle in Lemma~\ref{lem:discrete-maximum} therefore
gives
\[
 0\le\U^{n+1}\le G_{\delta,h}\U^n,
\]
which is \eqref{eq:comparison-G}.

Furthermore, $C_0\one\ge\one$.  If $Z=C_0^{-1}\one$, then $C_0(\one-Z)\ge0$, so the same maximum principle yields $0\le Z\le\one$.
For any vector $F$,
\[
 |C_0^{-1}F|
 \le C_0^{-1}|F|
 \le \norm{F}_\infty C_0^{-1}\one
 \le \norm{F}_\infty\one,
\]
and hence
\[
 \norm{C_0^{-1}}_\infty\le1.
\]
If $\beta=\max\{\norm{r}_\infty,\norm{s}_\infty\}$, then
\[
 \norm{I+\Delta tB_h}_\infty
 \le1+\beta\Delta t,
 \qquad
 \norm{G_{\delta,h}^n}_\infty
 \le(1+\beta\Delta t)^n
 \le e^{\beta T}
\]
for $n\Delta t\le T$.
\end{proof}

\begin{proof}[Proof of Theorem~\ref{thm:threshold-equivalence}]
Let $C_{\Delta t}=I-\Delta tA_{\delta,h}$.  By Lemma~\ref{lem:discrete-maximum}, $C_{\Delta t}^{-1}$ preserves
nonnegativity.  From \eqref{eq:discrete-principal-pair},
\[
 L_{\delta,h}\PhiV_h=\lambda_{\delta,h}\PhiV_h,
 \qquad \PhiV_h\gg0.
\]
Using $L_{\delta,h}=A_{\delta,h}+B_h$ and
\eqref{eq:linearized-step},
\begin{align*}
 G_{\delta,h}\PhiV_h-\PhiV_h
 &=C_{\Delta t}^{-1}
   \big[(I+\Delta tB_h)-C_{\Delta t}\big]\PhiV_h =\Delta t\,C_{\Delta t}^{-1}L_{\delta,h}\PhiV_h =\Delta t\,\lambda_{\delta,h}
   C_{\Delta t}^{-1}\PhiV_h.
\end{align*}
Set $W_h=C_{\Delta t}^{-1}\PhiV_h$.  Since $\PhiV_h\gg0$,
Lemma~\ref{lem:discrete-maximum} gives $W_h\gg0$.  Set
$m_h=\min_p(\PhiV_h)_p>0$.  Then every $X\in\mathbb R^{2N_h}$ satisfies
\[
 |X|\le \frac{\norm{X}_\infty}{m_h}\PhiV_h.
\]

If $\lambda_{\delta,h}>0$, then
\[
 G_{\delta,h}\PhiV_h
 \ge c_+\PhiV_h,
 \qquad
 c_+
 =1+\Delta t\,\lambda_{\delta,h}
   \min_p\frac{(W_h)_p}{(\PhiV_h)_p}>1.
\]
Because $G_{\delta,h}$ preserves order,
\[
 G_{\delta,h}^n\PhiV_h\ge c_+^n\PhiV_h,
\]
and hence $\rho(G_{\delta,h})\ge c_+>1$.

If $\lambda_{\delta,h}<0$, then $G_{\delta,h}\PhiV_h\ge0$ and
\[
 G_{\delta,h}\PhiV_h
 \le c_-\PhiV_h,
 \qquad
 c_-:=\max_p
 \frac{(G_{\delta,h}\PhiV_h)_p}{(\PhiV_h)_p}<1.
\]

For every $X\in\mathbb R^{2N_h}$,
\[
 |G_{\delta,h}^nX|
 \le G_{\delta,h}^n|X|
 \le
 \frac{\norm{X}_\infty}{m_h}c_-^n\PhiV_h.
\]
Consequently,
\[
 \norm{G_{\delta,h}^n}_\infty
 \le
 \frac{\norm{\PhiV_h}_\infty}{m_h}c_-^n,
\]
and the spectral-radius formula gives $\rho(G_{\delta,h})\le c_-<1$.

Finally, if $\lambda_{\delta,h}=0$, then
$G_{\delta,h}\PhiV_h=\PhiV_h$, so $1$ is an eigenvalue of
$G_{\delta,h}$. For any $X\in\mathbb R^{2N_h}$, positivity gives
\[
 |G_{\delta,h}^nX|
 \le
 G_{\delta,h}^n|X|
 \le
 \frac{\norm{X}_\infty}{m_h}G_{\delta,h}^n\PhiV_h
 =
 \frac{\norm{X}_\infty}{m_h}\PhiV_h.
\]
Thus
\[
 \norm{G_{\delta,h}^n}_\infty
 \le
 \frac{\norm{\PhiV_h}_\infty}{m_h}
 \qquad(n\ge1),
\]
so the spectral-radius formula gives $\rho(G_{\delta,h})\le1$.
Since $1$ is an eigenvalue, $\rho(G_{\delta,h})=1$.
\end{proof}

\begin{proof}[Proof of Corollary~\ref{cor:extinction}]
Lemma~\ref{lem:finite-bound} gives
\[
 0\le\U^n\le G_{\delta,h}^n\U^0.
\]
If $\lambda_{\delta,h}<0$, the proof of
Theorem~\ref{thm:threshold-equivalence} gives
\[
 G_{\delta,h}\PhiV_h\le c_-\PhiV_h,
 \qquad 0<c_-<1.
\]
Set $m_\Phi:=\min_p(\PhiV_h)_p>0$.  Since $\U^0\le(\norm{\U^0}_\infty/m_\Phi)\PhiV_h$, positivity of $G_{\delta,h}$ yields
\[
 0\le\U^n
 \le
 \frac{\norm{\U^0}_\infty}{m_\Phi}
 G_{\delta,h}^n\PhiV_h
 \le
 \frac{\norm{\U^0}_\infty}{m_\Phi}
 c_-^n\PhiV_h.
\]
Hence
\[
 \norm{\U^n}_\infty
 \le
 \frac{\norm{\PhiV_h}_\infty}{m_\Phi}
 c_-^n\norm{\U^0}_\infty.
\]
\end{proof}

\begin{proof}[Proof of Proposition~\ref{prop:positive-threshold}]
Part (i) follows from Theorem~\ref{thm:threshold-equivalence}.  If $\tau=0$, write
\[
 \mathcal F_h(\U)
 =
 L_{\delta,h}\U
 -
 \binom{
 \big(b_i^hU_{1,i}^2\big)_{i=1}^{N_h}}
 {
 \big(f_i^hU_{2,i}^2\big)_{i=1}^{N_h}}.
\]
Let $\U(t)$ and $\widetilde{\U}(t)$ be solutions with $0\le\U(0)\le\widetilde{\U}(0)$ and $\U(0)\ne\widetilde{\U}(0)$, and set $Z=\widetilde{\U}-\U$.  Then
\[
 \begin{aligned}
 Z_1'
 &=
 \left[
 \mu_1D_{\delta,h}-H_h
 -\operatorname{diag}\!\left(
 b_i^h(\widetilde U_{1,i}+U_{1,i})
 \right)_{i=1}^{N_h}
 \right]Z_1
 +R_hZ_2,\\
 Z_2'
 &=
 \left[
 \mu_2D_{\delta,h}-E_h
 -\operatorname{diag}\!\left(
 f_i^h(\widetilde U_{2,i}+U_{2,i})
 \right)_{i=1}^{N_h}
 \right]Z_2
 +S_hZ_1.
 \end{aligned}
\]
The strong-positivity argument of Corollary~\ref{cor:strong-positive-linear} therefore gives $\widetilde{\U}(t)-\U(t)\gg0$ for $t>0$.

Set
\[
 M(t)
 =
 \max\{\norm{U_1(t)}_\infty,\norm{U_2(t)}_\infty\},
 \qquad
 \beta_h=\max_i\{r_i^h,s_i^h\},
 \qquad
 \kappa_h=\min_i\{b_i^h,f_i^h\}>0.
\]
At a component attaining $M(t)$, the discrete dispersal term is
nonpositive, and hence
\[
 D^+M(t)
 \le
 \beta_hM(t)-\kappa_hM(t)^2.
\]
Consequently,
\[
 M(t)
 \le
 \max\left\{
 M(0),\frac{\beta_h}{\kappa_h}
 \right\}
 \qquad(t\ge0).
\]

Moreover, for $0<\theta<1$ and $\U\gg0$,
\[
 \mathcal F_h(\theta\U)-\theta\mathcal F_h(\U)
 =
 \theta(1-\theta)
 \binom{
 \big(b_i^hU_{1,i}^2\big)_{i=1}^{N_h}}
 {
 \big(f_i^hU_{2,i}^2\big)_{i=1}^{N_h}}
 \gg0.
\]
Thus the semiflow is strongly monotone, bounded for
nonnegative initial data, and strongly sublinear.  If
$\lambda_{\delta,h}>0$, write
$\PhiV_h=(\Phi_{1,h},\Phi_{2,h})^T\gg0$ for the principal eigenvector in
\eqref{eq:discrete-principal-pair}.  Then
\[
 \mathcal F_h(\varepsilon\PhiV_h)
 =
 \varepsilon\lambda_{\delta,h}\PhiV_h
 -\varepsilon^2
 \binom{
  \big(b_i^h(\Phi_{1,h})_i^2\big)_{i=1}^{N_h}}
  {\big(f_i^h(\Phi_{2,h})_i^2\big)_{i=1}^{N_h}}
 \gg0
\]
for all sufficiently small $\varepsilon>0$.  On the other hand, since
$\min_i\{b_i^h,f_i^h\}>0$ and $D_{\delta,h}\one=0$, there exists $K>0$
sufficiently large such that
\[
\mathcal F_h\!\left(K\binom{\one}{\one}\right)\ll0.
\]
These estimates, together with strong monotonicity, boundedness, and strong
sublinearity, give the hypotheses of the standard finite-dimensional
convergence theorem for monotone strongly sublinear semiflows; see
\cite{Smith1995}.
It follows that there is a unique equilibrium $\U_h^*\gg0$ satisfying
$\mathcal F_h(\U_h^*)=0$, and
\[
 \lim_{t\to\infty}\U(t;\U^0)=\U_h^*
 \qquad
 \text{for every }
 \U^0\in\R_+^{2N_h}\setminus\{0\}.
\]
Let $\Psi_{\Delta t}$ denote the numerical update map and $J_*$ the Jacobian of the semidiscrete vector field at $\U_h^*$.  Under the assumption in (iii), every eigenvalue of $J_*$ has negative real part.  Consistency of the semi-implicit step gives
\[
 D\Psi_{\Delta t}(\U_h^*)
 =
 I+\Delta tJ_*+O(\Delta t^2),
\]
so there exists $\Delta t_0>0$ such that all eigenvalues of $D\Psi_{\Delta t}(\U_h^*)$ lie inside the unit disk whenever $0<\Delta t<\Delta t_0$.  Hence $\U_h^*$ is a locally asymptotically stable fixed point of \eqref{eq:scheme} for every such $\Delta t$.
\end{proof}

\begin{proof}[Proof of Lemma~\ref{lem:residual-cont}]
Set $\eta=\eta(\PhiV,\zeta)$.  By definition,
\[
 (\zeta-\eta)\PhiV
 \le
 \cL_\delta\PhiV
 \le
 (\zeta+\eta)\PhiV.
\]
Since $\PhiV\in X_{++}\times X_{++}$, the Collatz--Wielandt characterization \eqref{eq:CW-cont} immediately yields
\[
 \zeta-\eta
 \le\lambda_\delta
 \le\zeta+\eta.
\]
\end{proof}

\begin{proof}[Proof of Corollary~\ref{cor:sign-interval}]
If $\lambda_h^+<0$, Lemma~\ref{lem:residual-cont} gives $\lambda_\delta<0$.  If $\lambda_h^->0$, it gives $\lambda_\delta>0$.
\end{proof}

\subsection{Reflected Cartesian local-limit results}

\begin{proof}[Proof of Theorem~\ref{thm:continuous-spectral-ac}]
The reflected operator $\cL_\delta^{\Box}$ generates a positive uniformly continuous semigroup on $C(\overline{\Omega_{\Box}})^2$.  Use $\PhiV^{0,\Box}$ as a positive test function.  Since the reaction matrix is unchanged,
\[
 \cL_\delta^{\Box}\PhiV^{0,\Box}
 -\lambda_0^{\Box}\PhiV^{0,\Box}
 =
 \begin{pmatrix}
 \mu_1(\cD_\delta^{\Box}\Phi_1^{0,\Box}
       -\Delta\Phi_1^{0,\Box})\\
 \mu_2(\cD_\delta^{\Box}\Phi_2^{0,\Box}
       -\Delta\Phi_2^{0,\Box})
 \end{pmatrix}.
\]
Lemma~\ref{lem:cont-local-consistency} bounds the numerator in the relative residual by $C\delta^2$.  Division by $m_\Phi$ gives
\[
 (\lambda_0^{\Box}-C\delta^2)\PhiV^{0,\Box}
 \le
 \cL_\delta^{\Box}\PhiV^{0,\Box}
 \le
 (\lambda_0^{\Box}+C\delta^2)\PhiV^{0,\Box}.
\]
Let $T_\delta^{\Box}(t)$ be the positive uniformly continuous semigroup generated by $\cL_\delta^{\Box}$.  The preceding inequalities imply
\[
 e^{(\lambda_0^{\Box}-C\delta^2)t}\PhiV^{0,\Box}
 \le
 T_\delta^{\Box}(t)\PhiV^{0,\Box}
 \le
 e^{(\lambda_0^{\Box}+C\delta^2)t}\PhiV^{0,\Box}.
\]
Since $\min_{x\in\overline{\Omega_\Box},\,i}\Phi_i^{0,\Box}(x)=m_\Phi>0$, every $\Psi\in X\times X$ satisfies
\[
 |\Psi|
 \le
 \frac{\norm{\Psi}_\infty}{m_\Phi}\PhiV^{0,\Box}.
\]
Set $a=e^{(\lambda_0^{\Box}-C\delta^2)t}$ and $b=e^{(\lambda_0^{\Box}+C\delta^2)t}$.  By positivity of $T_\delta^{\Box}(t)$, for every $n\ge1$,
\[
 a^n\PhiV^{0,\Box}
 \le
 T_\delta^{\Box}(t)^n\PhiV^{0,\Box}
 \le
 b^n\PhiV^{0,\Box}.
\]
Hence $\norm{T_\delta^{\Box}(t)^n}\ge a^n$.  Moreover, if $\norm{\Psi}_\infty\le1$, then
\[
 \begin{aligned}
 |T_\delta^{\Box}(t)^n\Psi|
 &\le
 T_\delta^{\Box}(t)^n|\Psi|\le
 \frac{1}{m_\Phi}
 T_\delta^{\Box}(t)^n\PhiV^{0,\Box}\le
 \frac{b^n}{m_\Phi}\PhiV^{0,\Box}.
 \end{aligned}
\]
Therefore
\[
 \norm{T_\delta^{\Box}(t)^n}
 \le
 \frac{\norm{\PhiV^{0,\Box}}_\infty}{m_\Phi}b^n.
\]
Taking $n$th roots and using the spectral-radius formula gives $a\le\rho(T_\delta^{\Box}(t))\le b$.  Since spectral mapping gives $\rho(T_\delta^{\Box}(t))=e^{t\lambda_\delta^{\Box}}$, we obtain
\[
 e^{t(\lambda_0^{\Box}-C\delta^2)}
 \le
 e^{t\lambda_\delta^{\Box}}
 \le
 e^{t(\lambda_0^{\Box}+C\delta^2)},
\]
which is \eqref{eq:continuous-spectral-ac}.
\end{proof}

\begin{proof}[Proof of Theorem~\ref{thm:discrete-spectral-ac}]
Set $\PhiV_h^{0,\Box}=\cR_h\PhiV^{0,\Box}$.  By Assumption~\ref{ass:eigenfunction}, $\PhiV_h^{0,\Box}\gg0$ and $\min_{i,\jj}(\PhiV_h^{0,\Box})_{i,\jj}\ge m_\Phi$.
Theorem~\ref{thm:discrete-local-consistency} gives
\[
 \left\|
 L_{\delta,h}^{\Box}\PhiV_h^{0,\Box}
 -\lambda_0^{\Box}\PhiV_h^{0,\Box}
 \right\|_\infty
 \le C(\delta^2+h^2).
\]
After division by the positive lower bound $m_\Phi$, this implies
\[
 (\lambda_0^{\Box}-C(\delta^2+h^2))\PhiV_h^{0,\Box}
 \le
 L_{\delta,h}^{\Box}\PhiV_h^{0,\Box}
 \le
 (\lambda_0^{\Box}+C(\delta^2+h^2))\PhiV_h^{0,\Box},
\]
with a possibly larger constant $C$.  The reflected discrete linear
evolution preserves nonnegativity; when the nonlocal formula is used,
positivity propagates through the positive coordinate-neighbor weights
in \eqref{eq:J-positive-ball}, while the local formula has the same
discrete maximum-principle property.  Hence, with
$T_{\delta,h}^{\Box}(t)=e^{tL_{\delta,h}^{\Box}}$,
\[
 e^{(\lambda_0^{\Box}-C(\delta^2+h^2))t}\PhiV_h^{0,\Box}
 \le
 T_{\delta,h}^{\Box}(t)\PhiV_h^{0,\Box}
 \le
 e^{(\lambda_0^{\Box}+C(\delta^2+h^2))t}\PhiV_h^{0,\Box}.
\]
Since $\min_{i,\jj}(\PhiV_h^{0,\Box})_{i,\jj}\ge m_\Phi>0$, every $X\in\R^{2N_h^{\Box}}$ satisfies
\[
 |X|
 \le
 \frac{\norm{X}_\infty}{m_\Phi}\PhiV_h^{0,\Box}.
\]
Set $a_h=e^{(\lambda_0^{\Box}-C(\delta^2+h^2))t}$ and $b_h=e^{(\lambda_0^{\Box}+C(\delta^2+h^2))t}$.  By positivity of $T_{\delta,h}^{\Box}(t)$, for every $n\ge1$,
\[
 a_h^n\PhiV_h^{0,\Box}
 \le
 T_{\delta,h}^{\Box}(t)^n\PhiV_h^{0,\Box}
 \le
 b_h^n\PhiV_h^{0,\Box}.
\]
Hence $\norm{T_{\delta,h}^{\Box}(t)^n}_\infty\ge a_h^n$.  If $\norm{X}_\infty\le1$, then
\[
 \begin{aligned}
 |T_{\delta,h}^{\Box}(t)^nX|
 &\le
 T_{\delta,h}^{\Box}(t)^n|X|\le
 \frac{1}{m_\Phi}
 T_{\delta,h}^{\Box}(t)^n\PhiV_h^{0,\Box}\le
 \frac{b_h^n}{m_\Phi}\PhiV_h^{0,\Box},
 \end{aligned}
\]
so
\[
 \norm{T_{\delta,h}^{\Box}(t)^n}_\infty
 \le
 \frac{\norm{\PhiV_h^{0,\Box}}_\infty}{m_\Phi}b_h^n.
\]
Taking $n$th roots gives $a_h\le\rho(T_{\delta,h}^{\Box}(t))\le b_h$.  The finite-dimensional spectral mapping theorem gives $\rho(T_{\delta,h}^{\Box}(t))=e^{t\lambda_{\delta,h}^{\Box}}$, and hence
\[
 |\lambda_{\delta,h}^{\Box}-\lambda_0^{\Box}|
 \le C(\delta^2+h^2),
\]
which is \eqref{eq:spectral-ac}.
\end{proof}

\begin{proof}[Proof of Corollary~\ref{cor:nonlocal-threshold-approx}]
The triangle inequality and Theorems~\ref{thm:continuous-spectral-ac} and \ref{thm:discrete-spectral-ac} give
\[
 |\lambda_{\delta,h}^{\Box}-\lambda_\delta^{\Box}|
 \le
 |\lambda_{\delta,h}^{\Box}-\lambda_0^{\Box}|
 +
 |\lambda_\delta^{\Box}-\lambda_0^{\Box}|
 \le
 C(h^2+\delta^2).
\]
\end{proof}

\begin{proof}[Proof of Corollary~\ref{cor:eventual-sign}]
Choose $h$ and $\delta$ so that $C(h^2+\delta^2)<\frac12|\lambda_0^{\Box}|$.  The conclusion follows from Theorem~\ref{thm:discrete-spectral-ac}.
\end{proof}

\begin{lemma}[Local consistency error for the reflected scheme]
\label{lem:local-residual}
Let $V^n=\cR_h\uvec^{0,\Box}(t_n)$ and define the consistency residual
$r^{n+1}$ by
\[
 \big[I-\Delta t\big(A_{\delta,h}^{\Box}-Q_h^{\Box}(V^n)\big)\big]V^{n+1}
 -
 \big(I+\Delta t B_h^{\Box}\big)V^n
 =
 \Delta t\,r^{n+1}.
\]
Under Assumption~\ref{ass:solution-reg},
\[
 \norm{r^{n+1}}_\infty
 \le
 C_T(\Delta t+h^2+\delta^2)
\]
for $t_{n+1}\le T$, with a constant independent of $h/\delta$.
\end{lemma}

\begin{proof}
After division by $\Delta t$, the backward time difference satisfies
\[
 \left\|
 \frac{V^{n+1}-V^n}{\Delta t}
 -
 \cR_h\partial_t\uvec^{0,\Box}(t_{n+1})
 \right\|_\infty
 \le C_T\Delta t.
\]
Theorem~\ref{thm:discrete-local-consistency} gives an
$O(h^2+\delta^2)$ bound for the dispersal error.  The explicit interstage
term contributes
$B_h^{\Box}(V^n-V^{n+1})=O(\Delta t)$, and the temporal regularity of the
local solution together with the Lipschitz continuity of $Q_h^{\Box}$ on
the relevant bounded set gives
\[
 \big(Q_h^{\Box}(V^{n+1})-Q_h^{\Box}(V^n)\big)V^{n+1}
 =
 O(\Delta t).
\]
Combining these estimates proves the stated bound.
\end{proof}

\begin{proof}[Proof of Theorem~\ref{thm:nonlinear-ac}]
Set $e^n=\U_{\Box}^n-\cR_h\uvec^{0,\Box}(t_n)$, $E_n=\norm{e^n}_\infty$, and $\varepsilon=\Delta t+h^2+\delta^2$.  Since $\U_{\Box}^0=\cR_h\uvec^{0,\Box}(0)$, we have $E_0=0$.  Write
$V^n=\cR_h\uvec^{0,\Box}(t_n)$ and set
\[
 C_n
 =I-\Delta t\big(A_{\delta,h}^{\Box}-Q_h^{\Box}(\U_{\Box}^n)\big).
\]
Subtracting the consistency identity in Lemma~\ref{lem:local-residual} from
the numerical step gives the exact error equation
\[
 \begin{aligned}
 C_n e^{n+1}
 &=
 (I+\Delta t B_h^{\Box})e^n \quad
 +\Delta t\big(Q_h^{\Box}(V^n)-Q_h^{\Box}(\U_{\Box}^n)\big)V^{n+1}
 -\Delta t\,r^{n+1}.
 \end{aligned}
\]
Because $Q_h^{\Box}(\U_{\Box}^n)$ is a nonnegative diagonal matrix,
$C_n$ has the same blockwise maximum-principle structure as the matrices in
Lemma~\ref{lem:discrete-maximum}.  In particular, $C_n^{-1}$ preserves
nonnegativity.  Moreover $C_n\one\ge\one$, and hence the comparison argument
used in Lemma~\ref{lem:finite-bound} gives
\[
 \norm{C_n^{-1}}_\infty\le1.
\]
The coefficient matrices in $B_h^{\Box}$ are uniformly bounded by the
continuous coefficients.  Since $Q_h^{\Box}$ is linear in its argument,
there is a constant $L_Q$, independent of $h$, $\delta$, and $\Delta t$, such
that
\[
 \norm{Q_h^{\Box}(V^n)-Q_h^{\Box}(\U_{\Box}^n)}_\infty
 \le L_Q E_n.
\]
Assumption~\ref{ass:solution-reg} gives
$\sup_{0\le t\le T}\norm{\uvec^{0,\Box}(t)}_\infty<\infty$, while
Lemma~\ref{lem:local-residual} gives
$\norm{r^{n+1}}_\infty\le C_T\varepsilon$.  Applying $C_n^{-1}$ to the
error equation therefore yields
\[
 E_{n+1}
 \le
 (1+C\Delta t)E_n
 +C\Delta t\,\varepsilon,
\]
where $C$ is independent of $h/\delta$.  Iteration gives
\[
 E_n
 \le
 C\Delta t\,\varepsilon
 \sum_{k=0}^{n-1}(1+C\Delta t)^k
 =
 \varepsilon\big[(1+C\Delta t)^n-1\big].
\]
Since $1+C\Delta t\le e^{C\Delta t}$ and $n\Delta t\le T$,
\[
 E_n
 \le
 (e^{CT}-1)\varepsilon
 \le
 C_T(\Delta t+h^2+\delta^2).
\]
Taking the maximum over $0\le n\Delta t\le T$ proves
\eqref{eq:nonlinear-error}.
\end{proof}

\section{Discussion and Concluding Remarks}
\label{sec:discussion-conclusion}
On general smooth habitats, the habitat-conforming discretization preserves
discrete mass, dissipation, nonnegativity, semidiscrete equilibria, and the
persistence--extinction threshold for each fixed interaction scale.  On
rectangular habitats, the reflected discretization additionally gives
second-order convergence of the discrete threshold in $h$ and $\delta$ and
finite-time solution error of order $\Delta t+h^2+\delta^2$, with constants
independent of $h/\delta$.  Thus the local-limit estimates require no prescribed
coupling between the mesh size and the interaction scale.

The reflection condition \eqref{eq:reflection-compatibility} makes explicit
the fourth-order boundary compatibility used in the local-limit estimates.
The numerical experiments confirm the conservation, threshold, nonnegativity,
and local-limit behavior established by the analysis.

\end{document}